\documentclass[a4paper,oneside,10pt, reqno]{amsproc}
\usepackage{amsthm}
\usepackage{amsfonts}
\usepackage{amsmath}
\usepackage{amssymb}
\usepackage{mathrsfs}
\usepackage{epsfig}
\usepackage{graphicx}
\usepackage{epstopdf}
\usepackage{tikz-cd}
\usepackage{color}
\usepackage{ifthen}
\usepackage{float}

\makeatletter
\@namedef{subjclassname@2010}{%
\textup{2010} Mathematics Subject Classification}
\makeatother

\newtheorem{theorem}{Theorem}[section]
\newtheorem{proposition}[theorem]{Proposition}
\newtheorem{corollary}[theorem]{Corollary}

\theoremstyle{remark}
\newtheorem{remark}[theorem]{Remark}

\theoremstyle{definition}
\newtheorem{example}[theorem]{Example}

\def \dep{\mathsf{d}}

\newcommand{\ran}{\mbox{ran}}

\newcommand{\bq}{\begin{equation}}
\newcommand{\eq}{\end{equation}}
\newcommand{\beqn}{\begin{eqnarray*}}
\newcommand{\eeqn}{\end{eqnarray*}}
\newcommand{\beq}{\begin{eqnarray}}
\newcommand{\eeq}{\end{eqnarray}}

\newcommand{\rar}{\rightarrow}

\newcommand{\bc}{\begin{centre}}
\newcommand{\ec}{\end{centre}}

\newcommand{\ba}{\begin{array}}
\newcommand{\ea}{\end{array}}

\newcommand{\inp}[2]{\langle{#1},\,{#2} \rangle}

\renewcommand{\Delta}{{\nabla}}

\newcommand*{\child}[1]{\mathsf{Chi}(#1)}
\newcommand*{\childn}[2]{{\mathsf{Chi}}^{\langle#1\rangle}(#2)}

\newcommand*{\Ge}{\geqslant}
\newcommand*{\lambdab}{\boldsymbol\lambda}
\newcommand*{\Le}{\leqslant}
\newcommand*{\parent}[1]{\mathsf{par}(#1)}

\begin{document}
\title[Operator-valued Multishifts]
{Unitary equivalence of operator-valued multishifts}
\author[R. Gupta]{Rajeev Gupta}
\author[S. Kumar]{Surjit Kumar}
\author[S. Trivedi]{Shailesh Trivedi}
\address{Department of Mathematics and Statistics\\
Indian Institute of Technology  Kanpur, India}
   \email{rajeevg@iitk.ac.in}
\address{Department of Mathematics \\ Indian Institute of Science Bangalore, India}
\email{surjitkumar@iisc.ac.in}
\address{Department of Mathematics and Statistics\\
Indian Institute of Technology  Kanpur, India}
   \email{shailtr@iitk.ac.in}

\thanks{The work of all the authors was supported by Inspire Faculty Fellowship (Ref. No. DST/INSPIRE/04/2017/002367, DST/INSPIRE/04/2016/001008, DST/INSPIRE/04/2018/000338 respectively).}

\date{}

\begin{abstract}
We systematically study various aspects of operator-valued multishifts. Beginning with basic properties, we show that the class of multishifts on the directed Cartesian product of rooted directed trees is contained in that of operator-valued multishifts. Further, we establish circularity, analyticity and wandering subspace property of these multishifts. In the rest part of the paper, we study the function theoretic behaviour of operator-valued multishifts. We determine the bounded point evaluation, reproducing kernel structure and the unitary equivalence of operator-valued multishifts with invertible operator weights. In contrast with a result of Lubin, it appears that the set of all bounded point evaluations  of  an operator-valued multishift may be properly contained in the joint point spectrum of the adjoint of underlying multishift. 
\end{abstract}

\maketitle

\section{Introduction}

Shift operators constitute a pivotal part of operator theory. The class of shift operators is rich enough to understand various notions in operator theory. There is tremendous literature on these operators (\cite{Gel1}, \cite{Gel2}, \cite{Gel3}, \cite{H}, \cite{S}, \cite{Gr}, \cite{Her}, \cite{N}, \cite{B}, \cite{MMN}) and equally huge literature exists on their generalized notions (\cite{L}, \cite{LT}, \cite{JL}, \cite{J}, \cite{JJS}, \cite{G}, \cite{MS}, \cite{CT}, \cite{B-D-P}, \cite{CPT}, \cite{CPT-1}). Two significant generalizations of classical weighted shift appeared in \cite{JL} and \cite{JJS}. The notion of shift operator introduced in \cite{JL} is commonly known as {\it classical multishift} which extends the notion of classical weighted shift into multivariable settings. On the other hand, the notion of shift operator introduced in \cite{JJS} is called {\it weighted shift on directed tree} which incorporates discrete structures (directed graphs) to give a broader picture of classical weighted shifts. The recently introduced notion of {\it multishifts on directed Cartesian product of rooted directed trees} in \cite{CPT} unifies both the aforementioned notions in the context of unilateral shifts. Another generalization of classical unilateral weighted shift namely, {\it operator-valued unilateral weighted shift}, appeared in \cite{L}. Hence one may ask whether there is any notion of shift which unifies all of the aforementioned notions of shift. In this regard, there comes the notion of {\it operator-valued multishifts}.
It turns out that the class of multishifts on directed Cartesian product of rooted directed trees is contained in that of operator-valued multishifts (see Proposition \ref{directed-tree}). Further, the class of operator-valued multishifts with invertible operator weights turns out to be a rich source of counter-examples to the von Neumann's inequality (see \cite{GKT}). One can not expect such phenomenon in the class of classical multishifts with non-zero weights (see \cite{Hz}). 

In this paper, we systematically study the various aspects of operator-valued multishifts. 
Starting with prerequisites and the formal definition, we discuss boundedness, commutativity and the moments of operator-valued multishifts. Further, we show that the class of multishifts on directed Cartesian product of rooted directed trees is contained in that of operator-valued multishifts. This constitutes the second section and a partial converse of the aforementioned inclusion concludes this section. In the third section, we give examples of various classes of operator-valued multishifts. We study circularity, analyticity and wandering subspace property of operator-valued multishifts in the fourth section. The last section deals with the study of function theoretic behaviour of operator-valued multishifts. In particular, we study the bounded point evaluation, reproducing kernel structure and the unitary equivalence of these multishifts. It turns out that the set of bounded point evaluations of operator-valued multishifts with invertible operator weights may be properly contained in the joint point spectrum of the adjoint of these multishifts (see Remark \ref{proper-inclusion}). This is in contrast with \cite[Proposition 19]{JL}.

We set below the notations used in posterior sections.
For a set $X$ and positive integer $d$,  $X^d$ stands for the $d$-fold Cartesian product of $X$.
The symbol ${\mathbb N}, \mathbb R$ and $\mathbb C$ stand for the set of non-negative
integers, field of the real numbers and the field of complex numbers, respectively.
For $\alpha =
(\alpha_1, \ldots, \alpha_d) \in {\mathbb{N}}^d,$
we set $|\alpha|:=\sum_{j=1}^d
\alpha_j$ and $\alpha ! := \prod_{j=1}^d \alpha_j !$. We follow the convention that $\alpha \in \mathbb N^d$ is always understood as $\alpha =(\alpha_1, \ldots, \alpha_d)$.
For $w=(w_1, \ldots, w_d) \in \mathbb C^d$ and $\alpha \in \mathbb N^d$, the complex conjugate $\overline{w} \in \mathbb C^d$ of $w$ is given by $(\overline{w}_1, \ldots, \overline{w}_d),$ while $w^\alpha$ denotes the complex number $\prod_{j=1}^d w^{\alpha_j}_j$. 
The symbol $\mathbb T^d$ denotes the $d$-torus in $\mathbb{C}^d.$
A subset $\Omega$ of $\mathbb C^d$ is said to have {\it polycircular symmetry} if for every $\lambda=(\lambda_1,\ldots,\lambda_d)\in\mathbb T^d$ and $z=(z_1,\ldots,z_d)\in\Omega$, $\lambda\cdot z:=(\lambda_1 z_1,\ldots, \lambda_d z_d)\in\Omega.$ A connected subset of $\mathbb C^d$ is called {\it Reinhardt} if it has polycircular symmetry. 
Let $\mathcal H$ be a complex Hilbert space.
If $F$ is a subset of $\mathcal H$, the closure of $F$ is denoted by $\overline{F}$, while the closed linear span of $F$ is denoted by $\bigvee \{x : x \in F\}$.  
If $\mathcal M$ is a subspace of $\mathcal H$, then $\dim \mathcal M$ denotes the Hilbert space dimension of $\mathcal M.$ 
Let $\mathcal{B}({\mathcal H})$ denote the unital Banach algebra of
bounded linear operators on $\mathcal H$ whereas $\mathcal G(\mathcal H) \subseteq \mathcal{B}({\mathcal H})$ denotes the set of invertible operators on $\mathcal H$. The multiplicative identity $I$ of 
$\mathcal{B}(\mathcal H)$ is sometimes denoted by $I_{\mathcal H}$.
The norm on  $\mathcal H$ is denoted by $\|\cdot\|_{\mathcal H}$ 
and whenever there is no confusion likely, 
we remove the subscript $\mathcal H$ from $\|\cdot\|_{\mathcal H}$. 
If $T \in \mathcal B(\mathcal H)$, then $\ker(T)$ denotes the kernel of $T$, the range of $T$ is denoted by $T(\mathcal H)$ or $\ran\, T$, $T^*$ denotes the Hilbert space adjoint of $T$ and $r(T)$ denotes the spectral radius of $T$. 
An operator $T \in \mathcal B(\mathcal H)$ is {\it left-invertible} if $T^*T$ is invertible in $\mathcal B(\mathcal H)$. 
We say that $T \in \mathcal B(\mathcal H)$ is
{\it analytic} if $\bigcap_{n \in \mathbb N}T^n(\mathcal H)=\{0\}$.
By a  {\it commuting $d$-tuple $T=(T_1, \ldots, T_d)$ in $\mathcal B(\mathcal H)$}, 
we mean a collection of commuting operators $T_1, \ldots, T_d$ in $\mathcal B(\mathcal H).$
For $\alpha  \in \mathbb N^d$, we understand $T^{\alpha}$ as the operator $T^{\alpha_1}_1\cdots T^{\alpha_d}_d$, where we adhere to the convention that $A^0 = I_\mathcal H$ for $A \in \mathcal \mathcal B(\mathcal H)$.
A commuting $d$-tuple $T=(T_1, \ldots, T_d)$ is said to be analytic if 
$\bigcap_{\alpha \in \mathbb N^d}T^\alpha(\mathcal H)=\{0\}$.
The notations $\sigma(T),\sigma_e(T),\sigma_l(T)$ and $\sigma_p(T)$ are reserved for the Taylor spectrum, 
joint essential spectrum, joint left-spectrum and joint point spectrum of a commuting $d$-tuple $T$ respectively.
The Hilbert space adjoint of the commuting $d$-tuple $T=(T_1, \ldots, T_d)$ is the $d$-tuple $T^*=(T^*_1, \ldots, T^*_d),$ and
the joint kernel $\bigcap_{j=1}^d \ker T_j$ of $T$ is denoted by $\ker T.$ 
A commuting $d$-tuple $T=(T_1, \ldots, T_d)$ is said to be toral left invertible 
if $T^*_j T_j$ is invertible for each $j=1,\ldots,d.$
Further, for $\lambda=(\lambda_1, \ldots, \lambda_d) \in \mathbb C^d$, by $T-\lambda$, we understand the $d$-tuple $(T_1-\lambda_1I_{\mathcal H}, \ldots, T_d-\lambda_dI_{\mathcal H})$. A commuting $d$-tuple $T=(T_1, \ldots, T_d)$ is said to be a {\it toral contraction} (resp. a {\it joint contraction}) if $T^*_j T_j \Le I$ for every $j=1, \ldots, d$ (resp. $\sum_{j=1}^d T^*_j T_j \Le I$). We say that $T=(T_1, \ldots, T_d)$ is a {\it row contraction} if $\sum_{j=1}^d T_j T^*_j \Le I$ and 
 it is said to be {\it joint expansion} if $\sum_{j=1}^d T^*_jT_j  \Ge I$.

\section{Operator-valued Multishift: Definition and basic properties}

Let $\{H_\alpha : \alpha \in \mathbb N^d\}$ be a multisequence of complex separable Hilbert spaces and let ${\mathcal H}=\oplus_{\alpha \in \mathbb N^d} H_\alpha$ be the orthogonal direct sum of $H_\alpha$, $\alpha \in \mathbb N^d$. Then $\mathcal H$ is a Hilbert space with respect to the following inner product:
\beqn \inp{x}{y}_{\mathcal H}= \sum_{\alpha \in \mathbb N^d} \inp{x_{\alpha}}{y_{\alpha}}_{H_\alpha},
\quad x=\oplus_{\alpha \in \mathbb N^d}x_{\alpha},\
y=\oplus_{\alpha \in \mathbb N^d} y_{\alpha} \in {\mathcal H}.\eeqn
If $H_\alpha = H$ for all $\alpha \in \mathbb N^d$, then we denote ${\mathcal H}=\oplus_{\alpha \in \mathbb N^d} H$ by $\ell^2_{H}(\mathbb N^d)$.
Let $\{{A^{(j)}_{\alpha}} : \alpha \in  \mathbb N^d,\ j = 1, \ldots, d \}$ be a multisequence of bounded linear operators $A^{(j)}_\alpha : H_\alpha \rar H_{\alpha+\varepsilon_j}.$
An {\it operator-valued multishift} $T$ on $\mathcal H = \oplus_{\alpha \in \mathbb N^d} H_\alpha$ with operator weights $\{{A^{(j)}_{\alpha}} : \alpha \in  \mathbb N^d,\ j = 1, \ldots, d \}$ is a $d$-tuple of operators $T_1, \ldots, T_d$ in $\mathcal H$ defined by
\beqn \mathcal D(T_j) :=
\Big \{\oplus_{\alpha \in \mathbb N^d }x_{\alpha} \in {\mathcal H} :
\sum_{\alpha \in \mathbb N^d} \|A^{(j)}_{\alpha}x_{\alpha} \|^2 < \infty\Big\},\eeqn

\beqn T_j(\oplus_{\alpha \in \mathbb N^d}x_{\alpha}):= \oplus_{\alpha \in \mathbb N^d}
A^{(j)}_{\alpha-\varepsilon_j}x_{\alpha-\varepsilon_j}, \quad \oplus_{\alpha \in \mathbb N^d}x_{\alpha}
\in \mathcal D(T_j), \ j=1, \ldots, d,\eeqn where $\varepsilon_j$ is the $d$-tuple in $\mathbb N^d$ with $1$ in the $j^{\mbox{\tiny{th}}}$ place and zeros
elsewhere. If for $\alpha \in \mathbb N^d$, $\alpha_j = 0$, then we interpret $A^{(j)}_{\alpha-\varepsilon_j}$ as a zero operator, $x_{\alpha-\varepsilon_j}$ as a zero vector and $H_{\alpha-\varepsilon_j} := \{0\}$. In what follows, if we encounter a situation where $\alpha_j$, in $(\alpha_1, \ldots, \alpha_d)$, is negative for some $j \in \{1, \ldots, d\}$, then   
we set $H_\alpha = \{0\}$,  $A^{(j)}_{\alpha} : H_{\alpha} \rar H_{\alpha+\varepsilon_j}$ to be a zero operator and $x_{\alpha}$ as a zero vector.

Note that each $T_j$, $j=1, \ldots, d$, is a densely defined linear operator in $\mathcal H$.  It follows from \cite[Proposition 3.1]{GKT} that $T_j$ is bounded if and only if
\beq \label{tj-bdd}
\sup_{\alpha \in \mathbb N^d} \|A^{(j)}_\alpha\| < \infty.
\eeq
Further, $T_i$ commutes with $T_j$ if and only if
\beq \label{commuting}
 A^{(i)}_{\alpha+\varepsilon_j} A^{(j)}_{\alpha}= A^{(j)}_{\alpha+\varepsilon_i} A^{(i)}_{\alpha} \ \mbox{for all } \alpha \in \mathbb N^d.
\eeq
That is, the following diagram commutes:
\[ \begin{tikzcd}
H_\alpha \arrow{r}{A^{(j)}_\alpha} \arrow[swap]{d}{A^{(i)}_{\alpha}} & H_{\alpha+\varepsilon_j} \arrow{d}{A^{(i)}_{\alpha+\varepsilon_j}} \\%
H_{\alpha+\varepsilon_i} \arrow{r}{\! A^{(j)}_{\alpha+\varepsilon_i}}&\ \ H_{\alpha+\varepsilon_j+\varepsilon_i}
\end{tikzcd}
\]
We refer to $T$ as {\it commuting operator-valued multishift} if the operator weights satisfy \eqref{tj-bdd} and \eqref{commuting}.
The proof of the following proposition is a routine verification. We leave it for the interested readers.

\begin{proposition}\label{kerT*}
Let $T = (T_1, \ldots, T_d)$ be a commuting operator-valued multishift on $\mathcal H = \oplus_{\alpha \in \mathbb N^d} H_\alpha$ with operator weights $\{A^{(j)}_{\alpha} : \alpha \in \mathbb N^d,\  j=1, \ldots, d\}$. Then 
\beqn T_j^*(\oplus_{\alpha
\in \mathbb N^d}x_{\alpha})= \oplus_{\alpha \in \mathbb N^d}
A^{(j)*}_{\alpha}x_{\alpha + \varepsilon_j}, \quad \oplus_{\alpha
\in \mathbb N^d} x_{\alpha} \in {\mathcal H},\ j=1, \ldots, d. \eeqn
Further, the joint kernel $\ker T^*$ of $T^*$ is given by
\beq\label{kerT*-eq}
\ker T^* = H_0 \oplus \bigoplus_{\alpha \in \mathbb N^d\setminus \{0\}} \ker A^*_{\alpha}, 
\eeq 
where $\ker A^*_{\alpha}=\bigcap_{j=1}^d \ker
A^{(j)*}_{\alpha-\varepsilon_j}$,  $\alpha \in \mathbb N^d\setminus \{0\}.$
\end{proposition}

The following proposition shows that there is no non-zero normal operator-valued multishift.

\begin{proposition}\label{Tj-normal}
Let $T = (T_1, \ldots, T_d)$ be a commuting operator-valued multishift on $\mathcal H = \oplus_{\alpha \in \mathbb N^d} H_\alpha$ with operator weights $\{A^{(j)}_{\alpha}: \alpha \in \mathbb N^d,\  j=1, \ldots, d\}$. Then for each $j = 1, \ldots, d$, $T_j$ is normal if and only if $T_j$ is zero.
\end{proposition}

\begin{proof}
Fix $j \in \{1, \ldots, d\}$. Clearly, if $T_j$ is zero, then it is normal. Suppose that $T_j$ is normal. Then for all $\oplus_{\alpha
\in \mathbb N^d} x_{\alpha} \in \mathcal H$, we get
\beqn
\oplus_{\alpha \in \mathbb N^d} A^{(j)*}_\alpha A^{(j)}_\alpha x_\alpha = \oplus_{\alpha \in \mathbb N^d} A^{(j)}_{\alpha-\varepsilon_j} A^{(j)*}_{\alpha-\varepsilon_j} x_\alpha,
\eeqn
which in turn implies that
\beq\label{normal-eq}
A^{(j)*}_\alpha A^{(j)}_\alpha = A^{(j)}_{\alpha-\varepsilon_j} A^{(j)*}_{\alpha-\varepsilon_j} \quad \mbox{for all } \alpha \in \mathbb N^d.
\eeq
It follows from \eqref{normal-eq} that for $\alpha = (\alpha_1, \ldots, \alpha_d) \in \mathbb N^d$, $A^{(j)}_\alpha = 0$ if $\alpha_j = 0$. To show that $A^{(j)}_\alpha = 0$ for all $\alpha \in \mathbb N^d$, we apply mathematical induction on $\alpha_j$. Clearly, the induction hypothesis holds for $\alpha_j = 0$. Suppose that for $k\in \mathbb{N},$ $A^{(j)}_\alpha = 0$ for all $\alpha \in \mathbb N^d$ with $\alpha_j =k$. Let $\beta \in \mathbb N^d$ be such that $\beta_j = k+1$. Then from \eqref{normal-eq}, we get
\beqn
A^{(j)*}_\beta A^{(j)}_\beta = A^{(j)}_{\beta-\varepsilon_j} A^{(j)*}_{\beta-\varepsilon_j}. 
\eeqn
Since $\beta_j - 1 = k$, it follows from the induction hypothesis that $A^{(j)}_{\beta-\varepsilon_j} = 0$ and hence $A^{(j)*}_\beta A^{(j)}_\beta = 0$ which in turn implies that $A^{(j)}_\beta = 0$. Thus $T_j$ must be zero.
\end{proof}

The following proposition gives explicit formula for the moments of a commuting operator-valued multishift. We will see later that these moments play a central role in the study of function theoretic behaviour of these multishifts. 

\begin{proposition}\label{moments}
Let $T = (T_1, \ldots, T_d)$ be a commuting operator-valued multishift on $\mathcal H = \oplus_{\alpha \in \mathbb N^d} H_\alpha$ with operator weights $\{A^{(j)}_{\alpha}: \alpha \in \mathbb N^d,\  j=1, \ldots, d\}$. For each $\alpha, \beta \in \mathbb N^d$, define
\beqn
B(\alpha, \beta) :=  A^{(1)}(\alpha,  \beta_1) A^{(2)}(\alpha-\beta_1\varepsilon_1,  \beta_2) \cdots A^{(d)}\Big(\alpha-\sum_{l=1}^{d-1}\beta_l\varepsilon_l,  \beta_d \Big),
\eeqn
 where  for each $j \in \{1, \ldots, d\}$, $A^{(j)}(\alpha, k) : H_{\alpha-k\varepsilon_j} \rar H_\alpha$ is given by
 \beqn
 A^{(j)}(\alpha, k) := \begin{cases}A^{(j)}_{\alpha-\varepsilon_j} \cdots A^{(j)}_{\alpha-k\varepsilon_j} & \mbox{if } k \Ge 1,\\
 I_{H_\alpha} & \mbox{if } k=0.
 \end{cases}
 \eeqn
Then for all $\alpha, \beta \in \mathbb N^d$, the following statements are true:
\begin{enumerate}
\item[(i)] $A^{(j)}_{\alpha} A^{(l)}(\alpha, k) = A^{(l)}(\alpha+\varepsilon_j, k) A^{(j)}_{\alpha-k\varepsilon_l}$\  \ for all\ \ $ j, l = 1, \ldots, d$ and $k \in \mathbb N$.
\item[(ii)] $A^{(j)}_{\alpha} B(\alpha, \beta) = B(\alpha+\varepsilon_j, \beta+\varepsilon_j)$ \  \ for all\ \   $j = 1, \ldots, d$.
\item[(iii)]
$
T^\beta (\oplus_{\alpha
\in \mathbb N^d} x_{\alpha}) = \oplus_{\alpha
\in \mathbb N^d} B(\alpha, \beta)\, x_{\alpha-\beta},
$ \ $\oplus_{\alpha
\in \mathbb N^d} x_{\alpha} \in \mathcal H$.
\item[(iv)] $
T^{*\beta} (\oplus_{\alpha
\in \mathbb N^d} x_{\alpha}) = \oplus_{\alpha
\in \mathbb N^d} C(\alpha, \beta)\, x_{\alpha+\beta},
$ \ $\oplus_{\alpha\in \mathbb N^d} x_{\alpha} \in \mathcal H$,
where
\beqn
\hspace{1.4cm} C(\alpha, \beta) = C^{(1)}(\alpha, \beta_1) C^{(2)}(\alpha+\beta_1\varepsilon_1, \beta_2) \cdots C^{(d)}\Big(\alpha+\sum_{l=1}^{d-1}\beta_l\varepsilon_l, \beta_d \Big)
\eeqn
and for $j \in \{1, \ldots, d\}$,
\beqn
C^{(j)}(\alpha, k) = \begin{cases}
A^{(j)*}_{\alpha} \cdots A^{(j)*}_{\alpha+(k-1)\varepsilon_j}  & \mbox{if } k \Ge 1,\\
I_{H_\alpha} & \mbox{if } k=0.
\end{cases}
\eeqn
\item[(v)] $
T^{*\beta} T^\beta (\oplus_{\alpha \in \mathbb N^d} x_{\alpha}) = \oplus_{\alpha \in \mathbb N^d} C(\alpha, \beta) B(\alpha+\beta, \beta)\, x_\alpha,$ \ $\oplus_{\alpha
\in \mathbb N^d} x_{\alpha} \in \mathcal H$.
\end{enumerate}
\end{proposition}

Note: For the sake of convenience, we write $B_\alpha$ in place of $B(\alpha, \alpha)$.

\begin{proof}
The equality in (i) follows from the repeated applications of commuting condition \eqref{commuting}. Indeed, if $k=0$, then the equality holds trivially for all $\alpha \in \mathbb N^d$. Let $\alpha \in \mathbb N^d$, $ j, l \in \{ 1, \ldots, d\}$ and $k \Ge 1$. Then
\beqn
A^{(j)}_\alpha A^{(l)}(\alpha, k) &=& A^{(j)}_\alpha  A^{(l)}_{\alpha-\varepsilon_l} \cdots A^{(l)}_{\alpha-k\varepsilon_l} \overset{\eqref{commuting}}= A^{(l)}_{\alpha+\epsilon_j-\epsilon_l}  A^{(j)}_{\alpha-\varepsilon_l} \cdots A^{(l)}_{\alpha-k\varepsilon_l} \\
&\overset{\eqref{commuting}}=&\cdots \overset{\eqref{commuting}} = A^{(l)}_{\alpha+\epsilon_j-\epsilon_l}  \cdots A^{(l)}_{\alpha+\varepsilon_j-k\varepsilon_l} A^{(j)}_{\alpha-k\varepsilon_l} = A^{(l)}(\alpha+\varepsilon_j, k) A^{(j)}_{\alpha-k\varepsilon_l}.
\eeqn

To see (ii), let $\alpha, \beta \in \mathbb N^d$ and $j \in \{1, \ldots, d\}$. Then
\beqn
A^{(j)}_\alpha B(\alpha, \beta) &=& A^{(j)}_\alpha A^{(1)}(\alpha, \beta_1) A^{(2)}(\alpha-\beta_1\varepsilon_1, \beta_2) \cdots A^{(d)}\Big(\alpha-\sum_{l=1}^{d-1}\beta_l\varepsilon_l, \beta_d \Big)\\
&\overset{\mbox{\tiny (i)}}=&  A^{(1)}(\alpha+\varepsilon_j, \beta_1) A^{(j)}_{\alpha-\beta_1\varepsilon_1}  \cdots A^{(d)}\Big(\alpha-\sum_{l=1}^{d-1}\beta_l\varepsilon_l, \beta_d \Big).
\eeqn
Continuing in this way, the right hand side of above expression becomes
\beqn
A^{(1)}(\alpha+\varepsilon_j, \beta_1) \cdots A^{(j)}\Big(\alpha-\sum_{l=1}^{j-1}\beta_l\varepsilon_l +\varepsilon_j,  \beta_j \Big) A^{(j)}_{\alpha-\sum_{l=1}^j \beta_l \varepsilon_l} \cdots A^{(d)}\Big(\alpha-\sum_{l=1}^{d-1}\beta_l\varepsilon_l, \beta_d \Big).
\eeqn
Note that
\beqn
A^{(j)}\Big(\alpha-\sum_{l=1}^{j-1}\beta_l\varepsilon_l +\varepsilon_j, \beta_j \Big) A^{(j)}_{\alpha-\sum_{l=1}^j \beta_l \varepsilon_l} = A^{(j)}\Big(\alpha-\sum_{l=1}^{j-1}\beta_l\varepsilon_l +\varepsilon_j, \beta_j+1 \Big).
\eeqn
Thus we get (ii).

We prove (iii) by induction on $|\beta|$, $\beta \in \mathbb N^d$. If $|\beta|=0$, then $\beta=0$ and hence the equality in (iii) holds trivially. Suppose that for some $n\in \mathbb N,$ it holds for all $\beta \in \mathbb N^d$ with $|\beta|=n$. Let $\gamma \in \mathbb N^d$ be such that $|\gamma|=n+1$. Then $\gamma=\beta+\varepsilon_j$ for some $j \in \{1, \ldots, d\}$ and $|\beta|=n$. Now for $\oplus_{\alpha \in \mathbb N^d} x_\alpha \in \mathcal H$, we get
\beqn
T^{\gamma}(\oplus_{\alpha \in \mathbb N^d} x_\alpha) &=& T_j T^{\beta}(\oplus_{\alpha \in \mathbb N^d} x_\alpha) =  T_j \big(\oplus_{\alpha \in \mathbb N^d} B(\alpha, \beta)\, x_{\alpha-\beta}\big) \\
&=& \oplus_{\alpha \in \mathbb N^d} A^{(j)}_{\alpha-\varepsilon_j} B(\alpha-\varepsilon_j, \beta)\, x_{\alpha-\varepsilon_j-\beta}\\
&\overset{\mbox{\tiny (ii)}}=& \oplus_{\alpha \in \mathbb N^d} B(\alpha, \gamma)\, x_{\alpha-\gamma}.
\eeqn
This completes the proof of (iii). The proof of (iv) goes along the lines of that of (iii) while (v) follows from (iii) and (iv).
\end{proof}

\begin{corollary}\label{dense-ker}
Let $T = (T_1, \ldots, T_d)$ be a commuting operator-valued multishift on $\mathcal H = \oplus_{\alpha \in \mathbb N^d} H_\alpha$ with operator weights $\{A^{(j)}_{\alpha}: \alpha \in \mathbb N^d,\  j=1, \ldots, d\}$. For $k \in \mathbb N$, let $D^{(k)} = (T_1^{*k}, \ldots, T_d^{*k})$. Then $\bigcup_{k \in \mathbb N} \ker D^{(k)}$ is a dense  subspace of $\mathcal H$.
\end{corollary}

\begin{proof}
Note that $\{\ker D^{(k)}\}_{k \in \mathbb N}$ is an increasing sequence of subspaces of $\mathcal H$ and hence $\cup_{k \in \mathbb N} \ker D^{(k)}$ is a subspace of $\mathcal H$. For each $\alpha \in \mathbb N^d$, define
$$M_\alpha := \big\{\oplus_{\beta \in \mathbb N^d}x_\beta \in \mathcal H : x_\beta = 0 \mbox{ if } \beta \neq \alpha \big\}.$$
Fix $k \geqslant 1$ and $j \in \{1, \ldots, d\}$. Let $\alpha \in \mathbb N^d$ be such that $|\alpha| <  k$. Then by substituting $\beta = k \varepsilon_j$ in Proposition \ref{moments}(iv), we see that $M_\alpha \subseteq  \ker T_j^{*k}$.
Consequently, $M_\alpha \subseteq  \ker D^{(k)}$. Thus $\bigcup_{k \in \mathbb N} \ker D^{(k)}$ is a dense  subspace of $\mathcal H$.
 \end{proof}

 \begin{remark}\label{point-spectrum-adjoint}
 It follows from the above corollary that $M_0 \subseteq \ker T^*$. Thus $0 \in \sigma_p(T^*) \subseteq \sigma(T^*)$.
 \end{remark}

It is noted in \cite[Pg. 7]{JJS} that using \cite[Proposition 2.1.12(vi), (viii)]{JJS} a weighted shift on an arbitrary directed tree can be realized as an operator-valued weighted shift with underlying directed tree being $\mathbb N,\ \mathbb Z,\ \mathbb N_{-}$ (set of all negative integers) or $\{0,\ldots,n\}$ for some $n\in\mathbb N$. 
Following the idea of \cite[Proposition 2.1.12]{JJS} and the notion of depth introduced in \cite[Definition 2.1.11]{CPT}, in what follows, it turns out that the class of commuting multishifts on directed Cartesian product of rooted directed trees lies in that of commuting operator-valued multishifts. 
The reader is referred to \cite[Chapters 2, 3]{CPT} for the definition and basic properties of multishifts on directed Cartesian product of rooted directed trees (see also \cite{JJS} for general theory of weighted shifts on directed trees).

\begin{proposition}\label{directed-tree}
Let $\mathscr T = (V,\mathcal E)$ be the directed Cartesian product of rooted directed trees $\mathscr T_1, \ldots, \mathscr T_d$ and let $S_{\lambdab} = (S_1, \ldots, S_d)$ be a commuting multishift on $\mathscr T$ with weights $\big\{\lambda^{(j)}_v : v \in V\setminus \{\mathsf{root}\},\ j = 1, \ldots, d \big\}$. Then $S_{\lambdab}$ is unitarily equivalent to a commuting operator-valued multishift.
\end{proposition}

\begin{proof}
For each $\alpha \in \mathbb N^d$, let
\beqn
V_\alpha := \{v \in V : \dep_v = \alpha\} \ \mbox{and } H_\alpha := \ell^2(V_\alpha),
\eeqn
where $\dep_v$ is the depth of $v$ in $\mathscr T$ (see \cite[Definition 2.1.11]{CPT}).
Note that $V=\bigsqcup_{\alpha\in \mathbb N^d} V_\alpha$ (see \cite[Lemma 2.1.10(vi)]{CPT}).
Set $\mathcal H := \oplus_{\alpha \in \mathbb N^d} H_\alpha$. Then there is a natural isometry from $\ell^2(V)$ into $\mathcal H$. Indeed, let $x = \sum_{v \in V} \inp{x}{e_v} e_v \in \ell^2(V)$, then
\beqn
x = \sum_{v \in V} \inp{x}{e_v} e_v = \sum_{\alpha \in \mathbb N^d}\sum_{v \in V_\alpha} \inp{x}{e_v} e_v = \sum_{\alpha \in \mathbb N^d} x_\alpha,
\eeqn
where $x_\alpha =\displaystyle\sum_{v \in V_\alpha} \inp{x}{e_v} e_v\in H_\alpha$. Here for $W \subseteq V$, we identify $\ell^2(W)$ as a subspace of $\ell^2(V)$. Now define $U : \ell^2(V) \rar \mathcal H$ by
\beqn
Ux = \oplus_{\alpha \in \mathbb N^d} x_\alpha, \quad x \in \ell^2(V). 
\eeqn
It is easy to see that $U$ is unitary. Now for each $\alpha \in \mathbb N^d$ and $j \in \{1, \ldots, d\}$, consider the linear operator $A^{(j)}_\alpha : H_\alpha \rar H_{\alpha+\varepsilon_j}$ given by
\beqn
A^{(j)}_\alpha x = S_j x, \quad x \in H_\alpha.
\eeqn
Clearly, $A^{(j)}_\alpha$ is bounded for each $\alpha \in \mathbb N^d$ and $j \in \{1, \ldots, d\}$. Let $T = (T_1, \ldots, T_d)$ be the operator-valued multishift on $\mathcal H$ with operator weights $\{A^{(j)}_\alpha : \alpha \in \mathbb N^d,\ j = 1, \ldots, d\}$. Since $S_{\lambdab}$ is commuting, it follows that $T$ is also commuting. Further, for $x \in \ell^2(V)$ and $j \in \{1, \ldots, d\}$, we have
\beqn
T_j Ux &=& T_j(\oplus_{\alpha \in \mathbb N^d} x_\alpha) = \oplus_{\alpha \in \mathbb N^d} A^{(j)}_{\alpha - \varepsilon_j} x_{\alpha - \varepsilon_j} = \oplus_{\alpha \in \mathbb N^d} S_j x_{\alpha - \varepsilon_j}\\
&=& U\Big(\sum_{\alpha \in \mathbb N^d} S_j x_{\alpha - \varepsilon_j}\Big) = US_j\Big(\sum_{\alpha \in \mathbb N^d} x_{\alpha - \varepsilon_j}\Big) = US_j x.
\eeqn
This completes the proof.
\end{proof}

A natural question arises here whether a commuting operator-valued multishift can be looked upon as a multishift on a directed Cartesian product of rooted directed trees. We do not know the answer of this question in more than one variable, but in case $d=1$, there are some sufficient conditions which ensure an affirmative answer of this question. 

\begin{proposition}\label{equi-shift}
Let $T$ be an operator-valued unilateral weighted shift on $\mathcal H = \oplus_{n \in \mathbb N} H_n$ with operator weights $\{A_n: n \in \mathbb N,\  j=1, \ldots, d\}$. Suppose that $\dim H_0 \geqslant 1$ and for each $n \in \mathbb N$, there exists an orthonormal basis $B_n$ of $H_n$ with the following property: 
For each $x \in B_n$ there exists a subset $W_x$ of $B_{n+1}$ such that  
\begin{enumerate}
\item[(a)] $B_{n+1} = \bigsqcup_{x \in B_n} W_x,$
\item[(b)] $\inp{A_n x}{y} \neq 0 \mbox{ for all } y \in W_x \mbox{ and }A_n x = \sum_{y \in W_x} \inp{A_n x}{y}y.$
\end{enumerate}
Then $T$ is unitarily equivalent to direct sum of $\dim H_0$ number of weighted shifts on rooted directed trees.
\end{proposition}

\begin{proof}
Set $V := \cup_{n \in \mathbb N} B_n$ and declare that for $x \in B_n$, $\child x = W_x$. This determines countably many disjoint rooted directed trees with roots being the elements of $B_0$. For $x \in B_n$, $n \geqslant 1$, set $\lambda_x := \inp{A_{n-1} \parent x}{x}_{\!_{H_n}}$ and consider the weight system $\lambda = \{\lambda_x : x \in V \setminus B_0\}$ of non-zero complex numbers. Let $\ell^2(V)$ be the Hilbert space of square summable complex-valued functions on $V$ with standard inner product. Note that the set $\{e_x : x \in V\}$ of characteristic functions of the singletons form an orthonormal basis of $\ell^2(V)$. Then the weighted shift operator $S_\lambda$ on $\ell^2(V)$ is a bounded linear operator. Let $U$ be the unitary operator from $\ell^2(V)$ onto $\mathcal H$ such that 
\beqn
U e_x = e_n \otimes x, \quad x \in V,
\eeqn
where $x \in B_n$, $n \in \mathbb N$, and $$e_n \otimes x := \oplus_{k=0}^\infty x_k =\begin{cases} x & \mbox{ if } k = n,\\ 0 & \mbox{ otherwise.} \end{cases}$$ Then for $x \in B_n$, $n \in \mathbb N$, we have
\beqn
U S_\lambda e_x &=& U \sum_{y \in \child x} \lambda_y e_y = \sum_{y \in W_x} \inp{A_n \parent y}{y} (e_{n+1} \otimes y)\\
&=& e_{n+1} \otimes \Big(\sum_{y \in W_x} \inp{A_n \parent y}{y} y\Big) = e_{n+1} \otimes A_n x = T U e_x.
\eeqn
This shows that $T$ is unitarily equivalent to $S_\lambda$. Further, it is easy to see that 
\beqn
V = \bigsqcup_{x \in B_0} V_x,
\eeqn
where for $x \in B_0$,
\beqn
V_x := \bigcup_{n \in \mathbb N} \childn{n}{x} = \{x\} \cup W_x \cup \Big( \bigcup_{y \in W_x} W_y\Big) \cup \cdots.
\eeqn
Thus we get $\ell^2(V) = \oplus_{x \in B_0} \ell^2(V_x)$ and each $\ell^2(V_x)$ is a reducing subspace of $S_\lambda$. Also note that for each $x \in B_0$, $V_x$ is a rooted directed tree with root $x$. Hence $S_\lambda = \oplus_{x \in B_0} S_x$, where $S_x := S_\lambda|_{\ell^2(V_x)}$ is a weighted shift on the rooted directed tree $V_x$. This completes the proof of the proposition.
\end{proof}

As an immediate consequence of the preceding proposition, we get the following corollary.

\begin{corollary}
Under the assumptions of Proposition \ref{equi-shift}, if $\dim H_0 = 1$, then $T$ is unitarily equivalent to a weighted shift on a rooted directed tree.
\end{corollary}

\section{Examples}

 We have already seen in Proposition \ref{directed-tree} that a large class of examples of operator-valued multishifts is supplied by the multishifts on directed Cartesian product of rooted directed trees. In this section, we exhibit few more classes of examples of commuting operator-valued multishifts. We begin with the following well-known example.

\begin{example}
Let $\{w^{(j)}_\alpha : \alpha \in \mathbb N^d,\ j=1, \ldots,d\}$ be a multisequence of non-zero complex numbers such that $\sup_{\alpha \in \mathbb N^d} |w^{(j)}_\alpha| < \infty$ and $w^{(j)}_{\alpha+\varepsilon_i} w^{(i)}_\alpha = w^{(i)}_{\alpha+\varepsilon_j} w^{(j)}_\alpha$ for all $\alpha \in \mathbb N^d$, $i,j=1,\ldots,d$. Let $\mathcal H = \ell^2_{\mathbb C}(\mathbb N^d)$. Set $A^{(j)}_\alpha := w^{(j)}_\alpha I_{\mathbb C}$ for all $\alpha \in \mathbb N^d$ and $j=1,\ldots,d$. Then the commuting operator-valued multishift $T=(T_1, \ldots, T_d)$ with operator weights $\{A^{(j)}_{\alpha}: \alpha \in \mathbb N^d,\  j=1, \ldots, d\}$ is commonly known as {\it classical multishift} \cite{JL}.
\end{example}

The following proposition gives a recipe to construct a number of examples of various classes of commuting operator-valued multishifts.

\begin{proposition}\label{example-prop}
Let $\{w^{(j)}_\alpha : \alpha \in \mathbb N^d,\ j=1, \ldots,d\}$ be a multisequence of non-zero complex numbers such that $\sup_{\alpha \in \mathbb N^d} |w^{(j)}_\alpha| < \infty$ and $w^{(j)}_{\alpha+\varepsilon_i} w^{(i)}_\alpha = w^{(i)}_{\alpha+\varepsilon_j} w^{(j)}_\alpha$ for all $\alpha \in \mathbb N^d$, $i,j=1,\ldots,d$. Let $H$ be a complex separable Hilbert space and let $\Phi : \mathbb N \rar \mathcal B(H)$ be a bounded function. For all $\alpha \in \mathbb N^d$ and $j \in \{1, \ldots, d\}$, set
\beqn A^{(j)}_{\alpha} := w^{(j)}_\alpha \Phi(|\alpha|).
\eeqn
Consider the operator-valued multishift $T = (T_1, \ldots, T_d)$ on $\mathcal H = \ell^2_{H}(\mathbb N^d)$ with operator weights $\{A^{(j)}_{\alpha}: \alpha \in \mathbb N^d,\  j=1, \ldots, d\}$. For all $\alpha \in \mathbb N^d$, set
\beqn
w_\alpha := \big(w^{(1)}_\alpha, \ldots, w^{(d)}_\alpha\big), \quad \tilde{w}_\alpha := \big(w^{(1)}_{\alpha-\varepsilon_1}, \ldots, w^{(d)}_{\alpha-\varepsilon_d}\big),
\eeqn
where we follow the convention that for $j \in \{1, \ldots, d\}$, $w^{(j)}_\alpha = 0$ whenever $\alpha_k < 0$ for some $k \in \{1, \ldots, d\}$. Then the following statements hold:
\begin{enumerate}
\item[(i)] $T$ is a commuting operator-valued multishift.
\item[(ii)] $T$ is a joint  contraction if and only if for all $\alpha \in \mathbb N^d$, the Hilbert space operator $w_\alpha \otimes \Phi(|\alpha|) : \mathbb C^d \otimes H \rar H$ is a contraction.
\item[(iii)] $T$ is a row contraction if and only if for all $\alpha \in \mathbb N^d$ with $|\alpha| \geqslant 1$, the Hilbert space operator $\tilde{w}_\alpha \otimes \Phi(|\alpha|-1) : \mathbb C^d \otimes H \rar H$ is a contraction.
\item[(iv)] $T$ is a joint expansion if and only if for all $\alpha \in \mathbb N^d$, the Hilbert space operator $w_\alpha \otimes \Phi(|\alpha|) : \mathbb C^d \otimes H \rar H$ is an expansion.
\end{enumerate}
\end{proposition}

\begin{proof}
It follows from the hypotheses that $\sup_{\alpha \in \mathbb N^d} \|A^{(j)}_\alpha\| < \infty$, $j = 1, \ldots, d$. For $i,j \in \{1, \ldots, d\}$ and $\alpha \in \mathbb N^d$, we have
\beqn
A^{(j)}_{\alpha+\varepsilon_i} A^{(i)}_{\alpha} = w^{(j)}_{\alpha+\epsilon_i} w^{(i)}_\alpha \Phi(|\alpha|+1)\Phi(|\alpha|)
= w^{(i)}_{\alpha+\varepsilon_j} w^{(j)}_\alpha \Phi(|\alpha|+1)\Phi(|\alpha|)
=A^{(i)}_{\alpha+\varepsilon_j} A^{(j)}_{\alpha}.
\eeqn
Now (i) follows from the commuting condition \eqref{commuting}.

Note that the commuting operator-valued multishift $T$ is a joint contraction if and only if  for all $\alpha \in \mathbb N^d$, $\sum_{j=1}^d A^{(j)*}_{\alpha} A^{(j)}_{\alpha} \Le I_{H}$. Fix $\alpha \in \mathbb N^d$. Then we get
\beqn
I_{H} \geqslant \sum_{j=1}^d A^{(j)*}_{\alpha} A^{(j)}_{\alpha} = \Big(\sum_{j=1}^d |w_\alpha^{(j)}|^2\Big) \Phi(|\alpha|)^* \Phi(|\alpha|) = \|w_\alpha\| \Phi(|\alpha|)^* \|w_\alpha\| \Phi(|\alpha|). 
\eeqn
Thus $ \|w_\alpha\| \|\Phi(|\alpha|)\| \leqslant 1$. Hence $w_\alpha \otimes \Phi(|\alpha|)$ is a contraction. This verifies (ii).

To see (iii), note that $T$ is a row contraction if and only if  for all $\alpha \in \mathbb N^d$, $\sum_{j=1}^d A^{(j)}_{\alpha-\varepsilon_j} A^{(j)*}_{\alpha-\varepsilon_j} \Le I_{H}$. Fix $\alpha \in \mathbb N^d$ with $|\alpha| \geqslant 1$. Then we get
\beqn
I_{H} &\geqslant& \sum_{j=1}^d A^{(j)}_{\alpha-\varepsilon_j} A^{(j)*}_{\alpha-\varepsilon_j} = \Big(\sum_{j=1}^d |w_{\alpha-\varepsilon_j}^{(j)}|^2\Big) \Phi(|\alpha|-1)^* \Phi(|\alpha|-1)\\
&=& \|\tilde{w}_\alpha\| \Phi(|\alpha|-1)^* \|\tilde{w}_\alpha\| \Phi(|\alpha|-1). 
\eeqn
Thus $ \|\tilde{w}_\alpha\| \|\Phi(|\alpha|-1)\| \leqslant 1$. Hence $\tilde{w}_\alpha \otimes \Phi(|\alpha|-1)$ is a contraction. This proves (iii). The proof of (iv) goes similarly along the lines of the proof of (ii).
\end{proof}

We below give an example of a commuting operator-valued multishifts with invertible operator weights which is a row contraction but not a joint contraction.

\begin{example}
Consider the system $\{A^{(j)}_{\alpha} : \alpha \in \mathbb N^d, \
j=1, \ldots, d\}$ of $2 \times 2$ matrices given by
\beqn A^{(j)}_{\alpha} &:=& \frac{1}{2}
\sqrt{\frac{\alpha_j+1}{|\alpha|+1}} \left(\begin{matrix}
\sqrt{\frac{|\alpha|+2}{|\alpha|+3}}+\sqrt{\frac{|\alpha|}{|\alpha|+1}}
&
\sqrt{\frac{|\alpha|+2}{|\alpha|+3}}-\sqrt{\frac{|\alpha|}{|\alpha|+1}}\\
&\vspace{-.3cm}\\
\sqrt{\frac{|\alpha|+2}{|\alpha|+3}}-\sqrt{\frac{|\alpha|}{|\alpha|+1}} &
\sqrt{\frac{|\alpha|+2}{|\alpha|+3}}+\sqrt{\frac{|\alpha|}{|\alpha|+1}} \end{matrix}\right),\quad \alpha \in \mathbb N^d \setminus \{0\},\\
A^{(j)}_0 &:=& \frac{1}{2\sqrt{3}} \left(\begin{matrix}
\sqrt{3}+1 & 1-\sqrt{3}\\
&\vspace{-.1cm}\\
1-\sqrt{3} & \sqrt{3}+1 \end{matrix}\right).
\eeqn
It is easy to see that the operator-valued multishift $T = (T_1, \ldots, T_d)$ on $\mathcal H = \ell^2_{\mathbb C^2}(\mathbb N^d)$ with operator weights $\{A^{(j)}_{\alpha}: \alpha \in \mathbb N^d,\  j=1, \ldots, d\}$  is commuting. Further, $T$ is a row contraction and $T$ is a joint contraction if and only if $d=1$.

To see this, first observe that 
\beqn
A^{(j)}_{\alpha} = w^{(j)}_\alpha \left(\begin{matrix}
\phi(|\alpha|)+\psi(|\alpha|) & \phi(|\alpha|)-\psi(|\alpha|)\\
&\vspace{-.1cm}\\
\phi(|\alpha|)- \psi(|\alpha|) & \phi(|\alpha|)+\psi(|\alpha|) \end{matrix}\right),
\eeqn 
where $w^{(j)}_\alpha := \frac{1}{2}
\sqrt{\frac{\alpha_j+1}{|\alpha|+1}}$, $\phi(n) := \sqrt{\frac{n+2}{n+3}}$ and $\psi(n) := \sqrt{\frac{n}{n+1}}$ for all $\alpha \in \mathbb N^d \setminus \{0\}$, $j \in \{1, \ldots, d\}$ and $n \in \mathbb N$. Now fix $\alpha \in \mathbb N^d$ with $|\alpha| \geqslant 1$. Then we have 
\beqn
\|\tilde{w}_\alpha\|^2 = \begin{cases} \frac{1}{4} \sum_{j=1}^d \frac{\alpha_j}{|\alpha|} = \frac{1}{4} & \mbox{if } |\alpha| \geqslant 2,\\ 
\frac{1}{(2\sqrt{3})^2} & \mbox{if } |\alpha|=1, \end{cases}
\  \quad \|\Phi(|\alpha|-1)\| = \begin{cases} 2 \phi(|\alpha|-1) & \mbox{if } |\alpha| \geqslant 2,\\
 2 \sqrt{3} & \mbox{if } |\alpha| =1.\end{cases}
\eeqn
It is immediate from above that $\|\tilde{w}_\alpha\| \|\Phi(|\alpha|-1)\|  < 1$. Hence it follows from Proposition \ref{example-prop}(iii) that $T$ is a row contraction.

To see the if and only if part, first note that if $d=1$, then the notion of row contraction coincides with that of joint contraction. Suppose that $T$ is a joint contraction. Note that for $\alpha \in \mathbb N^d \setminus \{0\}$,
 \beqn
\|w_\alpha\|^2 = \frac{1}{4} \sum_{j=1}^d  \frac{\alpha_j+1}{|\alpha|+1} = \frac{1}{4} \frac{|\alpha|+d}{|\alpha|+1}.
\eeqn
It follows from Proposition \ref{example-prop}(ii) that for $\alpha =\varepsilon_1$,
\beqn
\|w_{\varepsilon_1}\| \|\Phi(1)\| = \frac{\sqrt{d+1}}{2\sqrt{2}}\, 2\, \phi(1) = \frac{\sqrt{d+1}}{\sqrt{2}} \frac{\sqrt{3}}{\sqrt{4}} \leqslant 1.
\eeqn
Thus  $d \leqslant \frac{5}{3} < 2$. Since $d$ is a positive integer, we conclude that $d$ must be $1$. This completes the proof of aforementioned assertions.
\end{example}

\begin{remark}
The operator-valued multishift  discussed in the foregoing example may be realized as a  $d$-tuple of operators of multiplication by the coordinate functions on a reproducing kernel Hilbert space $\mathcal H(\kappa)$ of $\mathbb C^2$-valued holomorphic functions on the open unit ball $\mathbb B^d$ in $\mathbb C^d.$
Indeed, the positive definite kernel $\kappa$ is given by
\beqn \kappa(z,w) =\begin{pmatrix} \frac{1}{(1-\langle{z, w \rangle})^2} & \frac{\langle{z, w \rangle}}{(1-\langle{z, 
w \rangle})}\\ \frac{\langle{z, w \rangle}}{(1-\langle{z, w \rangle})} & \frac{1}{(1-\langle{z, w 
\rangle})^2}\end{pmatrix} = \begin{pmatrix} 1 & 0 \\ 0 & 1 \end{pmatrix}
+ \sum_{\substack{\alpha \in \mathbb N^d \\ |\alpha| \geq 1}} \frac{|\alpha|!}{\alpha !}\begin{pmatrix} |\alpha|+1 & 
1\\ 1 & |\alpha|+1 \end{pmatrix} z^{\alpha} \bar{w}^{\alpha}\eeqn
for all $z $ and $w$ in $\mathbb B^d$.
To see this, first note that if $\kappa_1$ and $\kappa_2$ are two positive definite kernels such that $\kappa_2 \leq \kappa_1$, then $\kappa$ given by
\beqn \kappa :=\begin{pmatrix} \kappa_1 & \kappa_2\\ \kappa_2 & \kappa_1 \end{pmatrix}\eeqn 
 is a positive definite kernel. Now consider the positive definite kernels $\kappa_1$ and $\kappa_2$ given by
\beqn \kappa_1(z,w) = \frac{1}{(1-\langle{z, w \rangle})^2}\ \mbox{ and } 
\ \kappa_2(z,w) = \frac{\langle{z, w \rangle}}{(1-\langle{z, w \rangle})} \mbox{ for all }
z ,w \in \mathbb B^d.\eeqn 
Then for all $z, w \in \mathbb B^d,$
\beqn \kappa(z,w) =\begin{pmatrix} \frac{1}{(1-\langle{z, w \rangle})^2} & \frac{\langle{z, w \rangle}}{(1-\langle{z, 
w \rangle})}\\ \frac{\langle{z, w \rangle}}{(1-\langle{z, w \rangle})} & \frac{1}{(1-\langle{z, w 
\rangle})^2}\end{pmatrix} = \begin{pmatrix} 1 & 0 \\ 0 & 1 \end{pmatrix}
+ \sum_{\substack{\alpha \in \mathbb N^d \\ |\alpha| \geq 1}} \frac{|\alpha|!}{\alpha !}\begin{pmatrix} |\alpha|+1 & 
1\\ 1 & |\alpha|+1 \end{pmatrix} z^{\alpha} \bar{w}^{\alpha}\eeqn
is a positive definite kernel. Consider the multisequence $\mathscr B = \{B_\alpha \in M_2(\mathbb C) : \alpha \in \mathbb N^d\}$ given by
\beqn  B_{\alpha} = \begin{cases}
 \begin{pmatrix} 1 & 0 \\ 0 & 1 \end{pmatrix}  & \mbox{ if } \alpha =0,\\ \notag
  \sqrt{\frac{|\alpha|!}{\alpha !}} \begin{pmatrix}  |\alpha|+1 & 1 \\ 1 & |\alpha|+1 \end{pmatrix}^{1/2} & \mbox{ if } |\alpha| \geq 1.
\end{cases}
\eeqn
Then we get
\beqn
B_\alpha = \begin{cases} \begin{pmatrix} 1 & 0 \\ 0 & 1 \end{pmatrix}  & \mbox{ if } \alpha =0,\\
\frac{1}{2}\sqrt{\frac{|\alpha|!}{\alpha !}}  \begin{pmatrix} 
\sqrt{|\alpha|+2}+\sqrt{|\alpha|} & \sqrt{|\alpha|+2}-\sqrt{|\alpha|}\\ 
& \vspace{-.3cm}&\\
\sqrt{|\alpha|+2}-\sqrt{|\alpha|} & 
\sqrt{|\alpha|+2}+\sqrt{|\alpha|} \end{pmatrix} & \mbox{ if } |\alpha| \geqslant 1.\end{cases}
\eeqn
Thus the $d$-tuple of operators of multiplication by the coordinate functions on the reproducing kernel Hilbert space $\mathcal H(\kappa)$ associated with positive definite kernel $\kappa$ may be realized as the commuting operator-valued multishift with operator weights $\{A^{(j)}_{\alpha}= B_{\alpha} B_{\alpha+\varepsilon_j}^{-1}: \alpha \in \mathbb N^d, \  j=1, \ldots, d\}$ (see Remark \ref{rm5.2}).
Hence,
\beqn A^{(j)}_{\alpha} &=& \frac{1}{2} \sqrt{\frac{\alpha_j+1}{|\alpha|+1}} \left(\begin{matrix} 
\sqrt{\frac{|\alpha|+2}{|\alpha|+3}}+\sqrt{\frac{|\alpha|}{|\alpha|+1}} & 
\sqrt{\frac{|\alpha|+2}{|\alpha|+3}}-\sqrt{\frac{|\alpha|}{|\alpha|+1}}\\ 
& \vspace{-.1cm}\\
\sqrt{\frac{|\alpha|+2}{|\alpha|+3}}-\sqrt{\frac{|\alpha|}{|\alpha|+1}} & 
\sqrt{\frac{|\alpha|+2}{|\alpha|+3}}+\sqrt{\frac{|\alpha|}{|\alpha|+1}} \end{matrix}\right), \quad \alpha \in \mathbb N^d \setminus \{0\},\\
A^{(j)}_0 &=& \frac{1}{2\sqrt{3}} \left(\begin{matrix}
\sqrt{3}+1 & 1-\sqrt{3}\\
&\vspace{-.1cm}\\
1-\sqrt{3} & \sqrt{3}+1 \end{matrix}\right).
\eeqn
\end{remark}

As in \cite[Sections 5.3, 5.4]{CPT}, the joint subnormal, joint hyponormal and joint $m$-isometric operator-valued multishifts can be characterized analogously. 

\section{Strong Circularity, Analyticity and Wandering Subspace Property}

A commuting $d$-tuple $S=(S_1, \ldots, S_d)$ on a Hilbert space $\mathcal H$ is said to be {\it circular} if for every $\lambda=(\lambda_1, \ldots, \lambda_d) \in \mathbb T^d$, there exists a unitary operator $U_\lambda$ on $\mathcal H$ such that $$U^*_\lambda S_j U_\lambda = \lambda_j S_j\ \text{for all}\ j=1, \ldots, d.$$
Further, $S$ is said to be {\it strongly circular} if $U_\lambda$ can be chosen to be a strongly continuous unitary representation of $\mathbb T^d$ in the following sense: For every $h \in \mathcal H$, the function
$\lambda \mapsto U_\lambda h$ is continuous on $\mathbb T^d$.

The above notion for $d=1$ was introduced and studied in \cite{AHHK}. Since then these operators became a centre of attraction and were studied considerably thereafter (refer to \cite{S}, \cite{Ge}, \cite{Ml}, \cite{BM-1}). The circularity of a classical multishift was first obtained in \cite[Corollary 3]{JL}. 
The circularity of a weighted shift on a directed tree was established in \cite[Theorem 3.3.1]{JJS}.
In view of Proposition \ref{directed-tree}, the following result generalizes \cite[Proposition 3.2.1]{CPT} and a special case of \cite[Theorem 3.3.1]{JJS}.

\begin{proposition}\label{circular}
Let $T = (T_1, \ldots, T_d)$ be a commuting operator-valued multishift on $\mathcal H = \oplus_{\alpha \in \mathbb N^d} H_\alpha$ with operator weights $\{A^{(j)}_{\alpha}: \alpha \in \mathbb N^d,\  j=1, \ldots, d\}$.  Then $T$ is strongly circular.
\end{proposition}

\begin{proof}
For $\lambda = (\lambda_1, \ldots, \lambda_d) \in \mathbb T^d$, define the linear operator $U_\lambda$ on ${ \mathcal H}$ by
\beqn
U_\lambda(\oplus_{\alpha \in \mathbb N^d}
x_{\alpha}) = \oplus_{\alpha \in \mathbb N^d}\overline{\lambda}^\alpha x_{\alpha}, \quad \oplus_{\alpha \in \mathbb N^d}x_{\alpha}
\in \mathcal H.
\eeqn
Clearly, $U_\lambda$ is a unitary operator on ${\mathcal H}$ and its adjoint is given by
\beqn
U^*_\lambda(\oplus_{\alpha \in \mathbb N^d}
x_{\alpha}) = \oplus_{\alpha \in \mathbb N^d}\lambda^\alpha x_{\alpha}, \quad \oplus_{\alpha \in \mathbb N^d}x_{\alpha}
\in \mathcal H.
\eeqn
For $\oplus_{\alpha \in \mathbb N^d}x_{\alpha}\in \mathcal H$ and $j \in \{1, \ldots, d\}$, we have
\beqn
U^*_\lambda T_j U_\lambda(\oplus_{\alpha \in \mathbb N^d}x_{\alpha})&=& U^*_\lambda T_j(\oplus_{\alpha \in \mathbb N^d}\overline{\lambda}^{\alpha}x_{\alpha})
= U^*_\lambda (\oplus_{\alpha \in \mathbb N^d}\overline{\lambda}^{\alpha-\varepsilon_j} A^{(j)}_{\alpha-\varepsilon_j}
x_{\alpha-\varepsilon_j})\\
&=& \oplus_{\alpha \in \mathbb N^d}\lambda^\alpha \overline{\lambda}^{\alpha-\varepsilon_j} A^{(j)}_{\alpha-\varepsilon_j} x_{\alpha-\varepsilon_j} = \oplus_{\alpha \in \mathbb N^d} \lambda_j A^{(j)}_{\alpha-\varepsilon_j} x_{\alpha-\varepsilon_j}\\
&=& \lambda_j T_j(\oplus_{\alpha \in \mathbb N^d}x_{\alpha}).
\eeqn
This proves that $T$ is circular. To see the strong circularity, let $\{\lambda^{(n)}\}_{n \in \mathbb N}$ be a sequence in $\mathbb T^d$ which converges to $\lambda \in \mathbb T^d$. Then for $\oplus_{\alpha \in \mathbb N^d}x_{\alpha} \in \mathcal H$, we get
\beqn
\|(U_{\lambda^{(n)}} - U_\lambda)\oplus_{\alpha \in \mathbb N^d}x_{\alpha}\|^2 = \sum_{\alpha \in \mathbb N^d} \big|\overline{\lambda^{(n)}}^\alpha - \overline{\lambda}^\alpha \big|^2 \|x_\alpha\|^2.
\eeqn
Let $\epsilon > 0$. Then there exists $n_1 \in \mathbb N$ such that $\sum_{\underset{|\alpha|>n_1}{\alpha \in \mathbb N^d}} \|x_\alpha\|^2 < \epsilon$. Further, for each $\alpha \in \mathbb N^d$, there exists $n_\alpha \in \mathbb N$ such that $\big|\overline{\lambda^{(n)}}^\alpha - \overline{\lambda}^\alpha \big|^2 < \epsilon$ for all $n \Ge n_\alpha$. Let $n_2 = \max\{n_\alpha : |\alpha| \Le n_1\}$. Then for all $n \Ge n_2$, we have
\beqn
\|(U_{\lambda^{(n)}} - U_\lambda)\oplus_{\alpha \in \mathbb N^d}x_{\alpha}\|^2 &=& \sum_{\alpha \in \mathbb N^d} \big|\overline{\lambda^{(n)}}^\alpha - \overline{\lambda}^\alpha \big|^2 \|x_\alpha\|^2\\
&=&\sum_{\underset{|\alpha|\Le n_1}{\alpha \in \mathbb N^d}} \big|\overline{\lambda^{(n)}}^\alpha - \overline{\lambda}^\alpha \big|^2 \|x_\alpha\|^2 + \sum_{\underset{|\alpha|> n_1}{\alpha \in \mathbb N^d}} \big|\overline{\lambda^{(n)}}^\alpha - \overline{\lambda}^\alpha \big|^2 \|x_\alpha\|^2\\
&<& \epsilon \sum_{\underset{|\alpha|\Le n_1}{\alpha \in \mathbb N^d}} \|x_\alpha\|^2 + 4 \epsilon.
\eeqn
This completes the proof.
\end{proof}

The following result is an immediate consequence of the preceding proposition.

\begin{corollary}\label{same-spectrum}
Let $T = (T_1, \ldots, T_d)$ be a commuting operator-valued multishift on $\mathcal H = \oplus_{\alpha \in \mathbb N^d} H_\alpha$ with operator weights $\{A^{(j)}_{\alpha}: \alpha \in \mathbb N^d,\  j=1, \ldots, d\}$. Then $\sigma(T),$ $ \sigma_p(T),$ $ \sigma_l(T)$ and $\sigma_e(T)$ have the polycircular symmetry. In particular, $\sigma(T)$ coincides with $\sigma(T^*)$.
\end{corollary}

The proof of the following proposition goes essentially along the lines of the proof of \cite[Proposition 3.2.4]{CPT}. The ideas involved in this proof are similar to that of  \cite[Lemma 3.8]{CY}. However, we include all the details for the sake of completeness.

\begin{proposition}\label{connected}
Let $T = (T_1, \ldots, T_d)$ be a commuting operator-valued multishift on $\mathcal H = \oplus_{\alpha \in \mathbb N^d} H_\alpha$ with operator weights $\{A^{(j)}_{\alpha}: \alpha \in \mathbb N^d,\  j=1, \ldots, d\}$. Then the Taylor spectrum of $T$ is connected.
\end{proposition}

\begin{proof}
In view of Corollary \ref{same-spectrum}, we prove that $\sigma(T^*)$ is connected. By Remark \ref{point-spectrum-adjoint}, $0 \in \sigma(T^*)$.  Let $F_1$ be the connected component of $\sigma(T^*)$ containing $0$ and let $F_2=\sigma(T^*) \setminus F_1.$ By the Shilov Idempotent Theorem
\cite[Application 5.24]{Cu}, there exist closed invariant subspaces
$\mathcal W_1, \mathcal W_2$ of $T^*$ such that $\mathcal H =
\mathcal W_1 \dotplus \mathcal W_2$ (vector space direct sum of $\mathcal W_1$ and $\mathcal W_2$) and $\sigma({T^*}|_{\mathcal W_l}) =F_l$
for $l=1, 2.$
For $k \in \mathbb N$, let $D^{(k)} = (T_1^{*k}, \ldots, T_d^{*k})$
and $h \in \ker D^{(k)}$. Then $h=x + y$ for $x \in
\mathcal W_1$ and $y \in \mathcal W_2.$ It follows that
$T^{*k}_{j} x=0=T^{*k}_{j}y$ for all $j=1, \ldots, d.$
If $y$ is nonzero, then
$0 \in \sigma_p(D^{(k)}|_{\mathcal W_2}) \subseteq \sigma(D^{(k)}|_{\mathcal W_2}),$
and hence by the spectral mapping property \cite{Cu}, $0 \in \sigma(D^{(1)}|_{\mathcal W_2})=\sigma({T^*}|_{\mathcal W_2}).$
Since $0 \notin F_2,$ we must have $y =0.$ It follows that $\mathcal
W_1$ contains the subspace $\cup_{k \in \mathbb N}
\ker D^{(k)},$ which is dense in $\mathcal H$ by Corollary \ref{dense-ker}.  Hence
$\mathcal W_1 = \mathcal H.$ Thus
the Taylor spectrum of $T^*$ is equal to $F_1$ and hence $\sigma(T^*)$ is connected.
\end{proof}

\begin{remark}
As a consequence of Corollary \ref{same-spectrum} and the preceding proposition, it follows that if $T = (T_1, \ldots, T_d)$ is a commuting operator-valued multishift on $\mathcal H = \oplus_{\alpha \in \mathbb N^d} H_\alpha$ with operator weights $\{A^{(j)}_{\alpha}: \alpha \in \mathbb N^d,\  j=1, \ldots, d\},$ then the Taylor spectrum of $T$ is Reinhardt.
\end{remark}

\begin{theorem}\label{ana-wand}
Let $T = (T_1, \ldots, T_d)$ be a commuting operator-valued multishift on $\mathcal H = \oplus_{\alpha \in \mathbb N^d} H_\alpha$ with operator weights $\{A^{(j)}_{\alpha}: \alpha \in \mathbb N^d,\  j=1, \ldots, d\}$. Then the following statements hold:
\begin{enumerate}
\item[(i)] $T$ is separately analytic. That is, for each $j \in \{1, \ldots, d\}$, $T_j$ is analytic.
\item[(ii)] $T$ has wandering subspace property. That is,
$$\bigvee_{\alpha \in \mathbb N^d} T^\alpha(\ker T^*) = \mathcal H.$$
\end{enumerate}
\end{theorem}

\begin{proof}
Let $j \in \{1, \ldots, d\}$. For each $k \Ge 1$, consider the subspace $M^{(j)}_k$ of $\mathcal H$ given by
\beqn
M^{(j)}_k = \big\{\oplus_{\alpha \in \mathbb N^d} x_\alpha : x_\alpha = 0 \mbox{ for }\alpha_j \Le k-1\big\}
\eeqn
and note that $\ran\, T^k_j \subseteq M^{(j)}_k$. Hence
\beqn
\bigcap_{k \Ge 1} \ran\, T^k_j \subseteq \bigcap_{k \Ge 1} M^{(j)}_k.
\eeqn
Now suppose that $x = \oplus_{\alpha \in \mathbb N^d} x_\alpha \in \displaystyle\bigcap_{k \Ge 1} M^{(j)}_k$. Then for every $\alpha \in \mathbb N^d$, $x_\alpha = 0$ as there is a $k \in \mathbb N$ bigger than $ \alpha_j$ such that $x \in M_k^{(j)}$. Hence $x=0$. Thus
\beqn
\bigcap_{k \Ge 1} \ran\, T^k_j \subseteq \bigcap_{k \Ge 1} M^{(j)}_k = \{0\}.
\eeqn
This completes the proof of (i).

To see (ii), let $\mathcal K = \bigvee_{\alpha \in \mathbb N^d} T^\alpha(\ker T^*)$. To show $\mathcal K = \mathcal H$, it is sufficient to show that $M_\alpha \subseteq \mathcal K$ for every $\alpha \in \mathbb N^d$, where
$$M_\alpha = \big\{\oplus_{\beta \in \mathbb N^d}x_\beta \in \mathcal H : x_\beta = 0 \mbox{ if } \beta \neq \alpha \big\}.$$
We prove by mathematical induction on $k \in \mathbb N$ that if $|\alpha| = k$, 
then $M_\alpha \subseteq \mathcal K$. 
Note that $\ker T^* \subseteq \mathcal K$ and hence it follows from \eqref{kerT*-eq} that 
$M_0 \subseteq \mathcal K$. Thus the induction hypothesis holds for $k = 0$. 
Suppose that it is true for some $k \in \mathbb N$. Let $\alpha \in \mathbb N^d$ be such that $|\alpha| = k+1$. 
Then by induction hypothesis $M_{\alpha-\varepsilon_j} \subseteq \mathcal K$ for all $j \in \{1, \ldots, d\}$. Since $\mathcal K$ is closed and $T_j$-invariant, it follows that
$$\overline{T_j M_{\alpha-\varepsilon_j}} = \Big\{\oplus_{\beta \in \mathbb N^d}x_\beta : x_\beta = 0 \mbox{ if } \beta \neq \alpha \mbox{ and } x_\alpha \in \overline{\ran\, A^{(j)}_{\alpha-\varepsilon_j}}\Big\}\subseteq \mathcal K.$$
Thus
\beq \label{K-1}
\Big\{\oplus_{\beta \in \mathbb N^d}x_\beta : x_\beta = 0 \mbox{ if } \beta \neq \alpha \mbox{ and } x_\alpha \in \sum_{j=1}^d\overline{\ran\, A^{(j)}_{\alpha-\varepsilon_j}}\Big\}\subseteq \mathcal K.
\eeq 
Further, from \eqref{kerT*-eq}, it follows that
\beq \label{K-2}
\Big\{\oplus_{\beta \in \mathbb N^d}x_\beta : x_\beta = 0 \mbox{ if } \beta \neq \alpha \mbox{ and } x_\alpha \in \bigcap_{j=1}^d \ker A^{(j)*}_{\alpha-\varepsilon_j}\Big\}\subseteq \ker T^* \subseteq \mathcal K.
\eeq
Combining \eqref{K-1} and \eqref{K-2}, we get that $M_\alpha  \subseteq \mathcal K.$
This completes the proof.
\end{proof}

Some immediate consequences of the preceding theorem are in order.

\begin{corollary}\label{T-analytic}
Let $T = (T_1, \ldots, T_d)$ be a commuting operator-valued multishift on $\mathcal H = \oplus_{\alpha \in \mathbb N^d} H_\alpha$ with operator weights $\{A^{(j)}_{\alpha}: \alpha \in \mathbb N^d,\  j=1, \ldots, d\}$. Then the following statements hold:
\begin{enumerate}
\item[(i)] $T$ is analytic.
\item[(ii)] The point spectrum $\sigma_p(T_j)$ of $T_j$ is contained in $\{0\}$ for each $j=1, \ldots, d$.
\item[(iii)] The joint point spectrum $\sigma_p(T)$ of $T$ is contained in $\{0\}$.
\item[(iv)] For each $j \in \{1, \ldots, d\}$,
\beqn
\bigvee_{k \in \mathbb N} \ker T^{*k}_j = \mathcal H = \bigvee_{\alpha \in \mathbb N^d} \ker T^{*\alpha}.
\eeqn
\item[(v)] For each $j \in \{1, \ldots, d\}$, the spectrum of $T_j$ is the closed disc of radius $r(T_j)$ centered at the origin, where $r(S)$ denotes the spectral radius of a bounded linear operator $S$.
\end{enumerate}
\end{corollary}

\begin{proof}
The proof of (i) follows from the preceding proposition and the fact that $\bigcap_{\alpha \in \mathbb N^d} \ran\, T^\alpha \subseteq \bigcap_{k \in \mathbb N} \ran\, T^k_j$ for all $j = 1, \ldots, d$.

To see (ii), fix $j \in \{1, \ldots, d\}$ and let $w \in \mathbb C$ be a non-zero eigenvalue of $T_j$. Then there exists a non-zero vector $x = \oplus_{\alpha \in \mathbb N^d}x_\alpha \in \mathcal H$ such that $T_j x = w x$. Hence $x \in \ran\, T^k_j$ for all $k \Ge 1$. This contradicts the Theorem \ref{ana-wand}(i) and thus we conclude that $\sigma_p(T_j) \subseteq \{0\}$. The proof of (iii) follows from (ii).

Fix $j \in \{1, \ldots, d\}$. Then from Theorem \ref{ana-wand}(i) and (i), it follows that
\beqn
\bigcap_{k \in \mathbb N} \ran\, T^k_j = \{0\} = \bigcap_{\alpha \in \mathbb N^d} \ran\, T^\alpha.
\eeqn
The desired conclusion now follows from taking the orthogonal complement on both sides of the above expression. This completes the verification of (iv).

Fix $j \in \{1, \ldots, d\}$. It follows from Proposition \ref{circular} that $\sigma(T_j)$ has circular symmetry. Further, proceeding along the lines of the proof of Proposition \ref{connected} and using (iv), one may deduce that $\sigma(T_j)$ is connected. Thus the spectrum of $T_j$ must be the closed disc of radius $r(T_j)$ centered at the origin. This verifies (v) and hence, completes the proof of the corollary.
\end{proof}

From Theorem \ref{ana-wand}(ii) and Proposition \ref{directed-tree}, we get the following corollary which extends \cite[Theorem 4.0.1]{CPT} into the settings of directed Cartesian product of finitely many rooted directed trees not necessarily locally finite.

\begin{corollary}
Let $\mathscr T = (V,\mathcal E)$ be the directed Cartesian product of rooted directed trees $\mathscr T_1, \ldots, \mathscr T_d$ and let $S_{\lambdab} = (S_1, \ldots, S_d)$ be a commuting multishift on $\mathscr T$ with weights $\big\{\lambda^{(j)}_v : v \in V\setminus \{\mathsf{root}\},\ j = 1, \ldots, d \big\}$. Then $S_{\lambdab}$ has wandering subspace property.
\end{corollary}

The following corollary gives a sufficient condition under which a commuting operator-valued multishift with non-invertible operator weights is unitarily equivalent to a commuting operator-valued multishift with invertible operator weights. It is an extended version of \cite[Corollary 4.1.12]{CPT} and its proof goes along the lines of that of \cite[Corollary 4.1.12]{CPT}. However, we include the details for the sake of completeness. 

\begin{corollary}
Let $T = (T_1, \ldots, T_d)$ be a toral left invertible commuting operator-valued multishift on $\mathcal H = \oplus_{\alpha \in \mathbb N^d} H_\alpha$ with operator weights $\{A^{(j)}_{\alpha}: \alpha \in \mathbb N^d,\  j=1, \ldots, d\}$. Suppose that $T^\alpha(\ker T^*)$ is orthogonal to $T^\beta(\ker T^*)$  whenever $\alpha \neq \beta$, $\alpha, \beta \in \mathbb N^d$. Then $T$ is unitarily equivalent to a commuting operator-valued multishift $\tilde{T} = (\tilde{T}_1, \ldots, \tilde{T}_d)$ on $\ell^2_{\ker T^*} (\mathbb N^d)$ with some invertible operator weights $\{\tilde{A}^{(j)}_{\alpha} \in \mathcal B(\ker T^*): \alpha \in \mathbb N^d,\  j=1, \ldots, d\}$.
\end{corollary}

\begin{proof}
As the spaces $T^\alpha(\ker T^*)$, $\alpha \in \mathbb N^d$, are mutually orthogonal, it follows from Theorem \ref{ana-wand}(ii) that
$\mathcal H = \oplus_{\alpha \in \mathbb N^d} T^\alpha(\ker T^*).$
Since $T$ is toral left invertible, it follows that for each $\alpha \in \mathbb N^d$, $\dim T^\alpha(\ker T^*) = \dim \ker T^*$ and $T^\alpha(\ker T^*)$ is a closed subspace of $\mathcal H$. Let $U_\alpha$ be a unitary from $T^\alpha(\ker T^*)$ onto $\ker T^*$. Now consider the operator $U : \mathcal H \rar \ell^2_{\ker T^*} (\mathbb N^d)$ given by
\beqn
U(\oplus_{\alpha \in \mathbb N^d} x_\alpha) = \oplus_{\alpha \in \mathbb N^d} U_\alpha x_\alpha, \quad x_\alpha \in T^\alpha(\ker T^*).
\eeqn
Then it is easy to see that $U$ is a unitary operator. For each $\alpha \in \mathbb N^d$ and $j \in \{1, \ldots, d\}$, let $\tilde{A}^{(j)}_{\alpha} := U_{\alpha+\varepsilon_j} T_j|_{T^\alpha(\ker T^*)} U_\alpha^*$. Clearly, each $\tilde{A}^{(j)}_{\alpha}$ is an invertible bounded linear operator on $\ker T^*$. Consider the operator-valued multishift $\tilde{T} = (\tilde{T}_1, \ldots, \tilde{T}_d)$ on $\ell^2_{\ker T^*} (\mathbb N^d)$ with operator weights $\{\tilde{A}^{(j)}_{\alpha} : \alpha \in \mathbb N^d,\  j=1, \ldots, d\}$.  Then for $\oplus_{\alpha \in \mathbb N^d} x_\alpha \in \mathcal H$, $x_\alpha \in T^\alpha(\ker T^*)$ and $j=1, \ldots, d$, we get
\beqn
UT_j(\oplus_{\alpha \in \mathbb N^d} x_\alpha) &=& U(\oplus_{\alpha \in \mathbb N^d} T_j x_\alpha) = U\big(\oplus_{\alpha \in \mathbb N^d} T_j|_{T^\alpha(\ker T^*)} x_\alpha\big)\\
&=& U\big(\oplus_{\alpha \in \mathbb N^d} T_j|_{T^{\alpha-\varepsilon_j}(\ker T^*)} x_{\alpha-\varepsilon_j}\big)\\
&=& \oplus_{\alpha \in \mathbb N^d} U_\alpha T_j|_{T^{\alpha-\varepsilon_j}(\ker T^*)} x_{\alpha-\varepsilon_j}\\
&=& \oplus_{\alpha \in \mathbb N^d} \tilde{A}^{(j)}_{\alpha-\varepsilon_j} U_{\alpha-\varepsilon_j} x_{\alpha-\varepsilon_j} = \tilde{T}_j U(\oplus_{\alpha \in \mathbb N^d} x_\alpha).
\eeqn
This completes the proof.
\end{proof}

\section{Operator-valued Multishifts with invertible operator weights}
This section is devoted to the study of function theory of operator-valued multishifts on $\ell^2_H(\mathbb N^d)$ with invertible operator weights, where $H$ is a complex separable HIlbert space. We begin with the observation that any such operator-valued multishift can be looked upon as a tuple of
operators of multiplication by the coordinate functions on a Hilbert space of vector-valued formal power series. In this
 regard, we recall some preliminaries about the Hilbert space of vector-valued formal power series.

Let $H$ be a complex separable Hilbert space. By an $H$-valued formal power series we mean the series
$\sum_{\alpha \in \mathbb N^d} x_\alpha z^\alpha$, $x_\alpha \in H$, without regard to convergence
at any point $z \in \mathbb C^d$. Fix a multisequence $\mathscr B = \{W_\alpha \in \mathcal G(H) : \alpha \in \mathbb N^d\}$. Then 
\beqn
\mathcal H^2(\mathscr B) := \Big\{\sum_{\alpha \in \mathbb N^d} x_\alpha z^\alpha : \sum_{\alpha \in \mathbb N^d} \|W_\alpha x_\alpha\|^2 < \infty\Big\}
\eeqn
 is a complex Hilbert space endowed with the following inner product:
  For $f(z) = \sum_{\alpha \in \mathbb N^d} x_\alpha z^\alpha$ and $g(z) = \sum_{\alpha \in \mathbb N^d} y_\alpha z^\alpha$ in $\mathcal H^2(\mathscr B)$,
 \beqn
 \inp{f(z)}{g(z)}_{\!_{\mathcal H^2(\mathscr B)}} := \sum_{\alpha \in \mathbb N^d}\inp{W_\alpha x_\alpha}{W_\alpha y_\alpha}_{\!_H}.
 \eeqn
We refer to $\mathcal H^2(\mathscr B)$ as a Hilbert space of $H$-valued formal power series. Note that  $\|xz^\alpha\|_{\mathcal H^2(\mathscr B)} = \|W_\alpha x\|_H$ for all $x \in H$ and $\alpha \in \mathbb N^d$.

Let $\mathscr M_z = (\mathscr M_{z_1}, \ldots, \mathscr M_{z_d})$ denote the $d$-tuple of operators of
 multiplication by the coordinate functions $z_j$
 defined on the subspace $\Big \{\displaystyle \sum_{\underset{|\alpha| \leqslant n}{\alpha \in \mathbb N^d}} x_\alpha z^\alpha : x_\alpha \in H,\ n \in \mathbb N \Big\}$ of polynomials in $\mathcal H^2(\mathscr B)$ by
\beqn
\mathscr M_{z_j} \Big(\sum_{\underset{|\alpha| \leqslant n}{\alpha \in \mathbb N^d}} x_\alpha z^\alpha \Big) = \sum_{\underset{|\alpha| \leqslant n}{\alpha \in \mathbb N^d}} x_\alpha z^{\alpha+\varepsilon_j},  \quad j=1, \ldots, d.
\eeqn
Thus $\mathscr M_{z_j}$ are densely defined linear operators on $\mathcal H^2(\mathscr B)$ which may not be bounded in general. The following proposition shows that an operator-valued multishift with invertible operator weights can be recognized as the $d$-tuple of operators of multiplication by the coordinate functions on some Hilbert space of formal power series.

\begin{theorem}\label{H2-beta-model}
Let $H$ be a complex separable Hilbert space and $T = (T_1, \ldots, T_d)$ be a commuting
operator-valued multishift on $\ell^2_H(\mathbb N^d)$ with invertible operator weights
$\{A^{(j)}_{\alpha} : \alpha \in \mathbb N^d,\  j=1, \ldots, d\}$.
Let $\mathscr B = \{B_\alpha \in \mathcal G(H) : \alpha \in \mathbb N^d\}$,
where $B_\alpha$ is as defined in Proposition \ref{moments} $($see note after the statement of Proposition \ref{moments}$)$. Then for each $j \in \{1, \ldots, d\}$, the operator $\mathscr M_{z_j}$ of multiplication by the coordinate function $z_j$ on $\mathcal H^2(\mathscr B)$ is bounded. Further, there exists a
 unitary $U:\ell^2_H(\mathbb N^d) \rar \mathcal H^2(\mathscr B)$ such that
 $\mathscr M_{z_j} U = U T_j$ for all $j = 1, \ldots, d$. Moreover, the following statements hold:
\begin{enumerate}
\item[(i)]  For all $\displaystyle\sum_{\alpha \in \mathbb N^d} x_\alpha z^\alpha \in \mathcal H^2(\mathscr B)$,
$$\displaystyle\mathscr M^*_{z_j} \Big(\sum_{\alpha \in \mathbb N^d} x_\alpha z^\alpha \Big) = \sum_{\alpha \in \mathbb N^d} B_\alpha^{-1} A_\alpha^{(j)*}
 B_{\alpha+\varepsilon_j} x_{\alpha+\varepsilon_j} z^\alpha.$$
 
  \item[(ii)] A point $w \in \sigma_p(\mathscr M_z^*)$ if and only if there exists a non-zero
 vector $x$ in $H$ such that $\displaystyle \sum_{\alpha \in \mathbb N^d} |w^\alpha|^2 \|B_\alpha^{*-1}x\|^2 < \infty$.
\end{enumerate}
\end{theorem}

\begin{proof}
It follows easily from \eqref{tj-bdd} and Proposition \ref{moments}(ii) that $\mathscr M_{z_j}$ is bounded on $\mathcal H^2(\mathscr B)$.
For each $\alpha \in \mathbb N^d$ and $j \in \{1, \ldots, d\}$, $A^{(j)}_{\alpha}$ is invertible so is $B_\alpha$. Thus
$B_\alpha H = H$ and  $\ell^2_H(\mathbb N^d) = \oplus_{\alpha \in \mathbb N^d} B_\alpha H$.
It follows that for each $x \in \ell^2_H(\mathbb N^d)$, there exists a unique sequence
$\{x_\alpha\}_{\alpha \in \mathbb N^d}$ in $H$ such that $x = \oplus_{\alpha \in \mathbb N^d} B_\alpha x_\alpha$.
Define the map $U:\ell^2_H(\mathbb N^d) \rar \mathcal H^2(\mathscr B)$ by
\beqn
U (\oplus_{\alpha \in \mathbb N^d} B_\alpha x_\alpha) = \sum_{\alpha \in \mathbb N^d} x_\alpha z^\alpha.
\eeqn 
It is easy to verify that $U$ is a surjective isometry and hence a unitary. Fix $j \in \{1, \ldots, d\}$. Then for $\oplus_{\alpha \in \mathbb N^d} B_\alpha x_\alpha \in \ell^2_H(\mathbb N^d)$ and using the Proposition \ref{moments}(ii), we have
\beqn
\mathscr M_{z_j} U \big(\oplus_{\alpha \in \mathbb N^d} B_\alpha x_\alpha\big) &=&
\mathscr M_{z_j}\Big(\sum_{\alpha \in \mathbb N^d} x_\alpha z^\alpha\Big) = \sum_{\alpha \in \mathbb N^d} x_\alpha z^{\alpha+\varepsilon_j}\\
&=& \sum_{\alpha \in \mathbb N^d} x_{\alpha-\varepsilon_j} z^{\alpha}
= U\big( \oplus_{\alpha \in \mathbb N^d} B_\alpha x_{\alpha-\varepsilon_j}\big)\\
&=& U\big( \oplus_{\alpha \in \mathbb N^d} A^{(j)}_{\alpha-\varepsilon_j}B_{\alpha-\varepsilon_j} x_{\alpha-\varepsilon_j}\big)\\
&=& U T_j \big(\oplus_{\alpha \in \mathbb N^d} B_\alpha x_\alpha\big).
\eeqn
This completes the proof of the first half of the proposition. For the moreover part,
let $\sum_{\alpha \in \mathbb N^d} x_\alpha z^\alpha \in \mathcal H^2(\mathscr B)$ and note that
\beqn
\mathscr M^*_{z_j}\Big(\sum_{\alpha \in \mathbb N^d} x_\alpha z^\alpha\Big)&=& UT_j^*U^*
 \Big(\sum_{\alpha \in \mathbb N^d} x_\alpha z^\alpha\Big) = UT_j^*(\oplus_{\alpha \in \mathbb N^d} B_\alpha x_\alpha)\\
&=& U \big(\oplus_{\alpha \in \mathbb N^d} A^{(j)*}_{\alpha} B_{\alpha+\varepsilon_j} x_{\alpha+\varepsilon_j} \big)\\
&=& U \big(\oplus_{\alpha \in \mathbb N^d} B_\alpha B_\alpha^{-1} A^{(j)*}_{\alpha} B_{\alpha+\varepsilon_j}  x_{\alpha+\varepsilon_j} \big)\\
&=& \sum_{\alpha \in \mathbb N^d} B_\alpha^{-1} A^{(j)*}_{\alpha} B_{\alpha+\varepsilon_j} x_{\alpha+\varepsilon_j} z^\alpha.
\eeqn
This establishes (i). 

To see (ii), let $w \in \sigma_p(\mathscr M_z^*)$ and $f_w(z) = \sum_{\alpha \in \mathbb N^d} x_\alpha z^\alpha
 \in \mathcal H^2(\mathscr B)$ be an eigenvector corresponding to $w$. Then 
\beq\label{eigen-eq} 
 \mathscr M^*_{z_j} f_w(z) = w_j f_w(z)\ \mbox{ for all } j = 1, \ldots, d.
\eeq 
Using (i) and comparing the coefficients of $z^\alpha$ on both sides of \eqref{eigen-eq}, we get
\beqn
B_\alpha^{-1} A^{(j)*}_{\alpha} B_{\alpha+\varepsilon_j} x_{\alpha+\varepsilon_j}
= w_j x_\alpha \mbox{ for all } \alpha \in \mathbb N^d \mbox{ and } j = 1, \ldots, d.
\eeqn
Consequently, $\mbox{ for all } \alpha \in \mathbb N^d \mbox{ and } j \in \{1, \ldots, d\},$
\beq\label{recurse}
x_{\alpha+\varepsilon_j} = w_j B_{\alpha+\varepsilon_j}^{-1} A^{(j)*-1}_{\alpha} B_\alpha x_\alpha.
\eeq
It is easy to see that for fixed $x_0 \in H$, $x_\alpha$ is well-defined for each $\alpha \in \mathbb N^d$. That is,
 $x_{(\alpha+\varepsilon_j)+\varepsilon_k} = x_{(\alpha+\varepsilon_k)+\varepsilon_j}$ for all
  $\alpha \in \mathbb N^d$ and $j, k = 1, \ldots, d$. By repeated applications of \eqref{recurse} and using Proposition \ref{moments}(ii), we get
\beqn
x_\alpha = w^\alpha (B_\alpha^{*} B_\alpha)^{-1}  x_0 \mbox{ for all } \alpha \in \mathbb N^d,
\eeqn
where $x_0 := f_w(0)$. Thus $f_w(z) = \displaystyle \sum_{\alpha \in \mathbb N^d} w^\alpha (B_\alpha^{*} B_\alpha)^{-1} x_0 z^\alpha$.
As $f_w \in \mathcal H^2(\mathscr B)$, we must have $\displaystyle \sum_{\alpha \in \mathbb N^d} |w^\alpha|^2 \|B_\alpha^{*-1}x_0\|^2 < \infty$. 

Conversely, suppose that $\displaystyle \sum_{\alpha \in \mathbb N^d} |w^\alpha|^2 \|B_\alpha^{*-1}x\|^2 < \infty$ for some $w \in \mathbb C^d$ and a non-zero vector $x \in H$. Define $$g(z) := \displaystyle \sum_{\alpha \in \mathbb N^d} w^\alpha (B_\alpha^{*} B_\alpha)^{-1} x z^\alpha.$$
Then $g$ is a non-zero element of $\mathcal H^2(\mathscr B)$. From (i) and Proposition \ref{moments}(ii), we get for all $j \in \{1, \ldots, d\}$,
\beqn
\mathscr M^*_{z_j} g &=& \sum_{\alpha \in \mathbb N^d} w^{\alpha +\varepsilon_j} B_\alpha^{-1} A^{(j)*}_{\alpha} B_{\alpha+\varepsilon_j} (B^*_{\alpha+\varepsilon_j} B_{\alpha+\varepsilon_j})^{-1} x z^\alpha\\
&=& w_j \sum_{\alpha \in \mathbb N^d} w^{\alpha} B_\alpha^{-1} B_\alpha^{*-1} x z^\alpha = w_j g.
\eeqn
This completes the proof of the proposition.
\end{proof}

\begin{remark} \label{rm5.2}
It can be easily verified that for a given multisequence $\{B_\alpha \in \mathcal G(H) : \alpha \in \mathbb N^d\}$
 of invertible operators on $H$, if the $d$-tuple $\mathscr M_z = (\mathscr M_{z_1}, \ldots, \mathscr M_{z_d})$
 of operators of multiplication by the coordinate functions on $\mathcal H^2(\mathscr B)$ is bounded, then it is
 unitarily equivalent to a commuting operator-valued multishift on $\ell^2_H(\mathbb N^d)$ with invertible
 operator weights $A^{(j)}_\alpha = B_{\alpha+\varepsilon_j} B_\alpha^{-1}, \alpha \in \mathbb N^d \mbox{ and } j = 1, \ldots, d$.
\end{remark}

A natural question arises here is whether the Hilbert space $\mathcal H^2(\mathscr B)$ obtained
 in the foregoing proposition can be realized as a reproducing kernel Hilbert space of holomorphic
 functions. Before proceeding towards an answer of this question, let us recall a definition  from \cite{JL}. For $w \in \mathbb C^d$, let $E_w$ denote the linear map of evaluation at $w$ defined on the polynomials in $\mathcal H^2(\mathscr B)$ onto $H$ by $E_w p = p(w)$. A point $w \in \mathbb C^d$ is called a {\it bounded point evaluation $($bpe$)$} on $\mathcal H^2(\mathscr B)$ if $E_w$ extends to a continuous linear map from $\mathcal H^2(\mathscr B)$ onto $H$. By an abuse of notation, we denote the continuous extension of $E_w$ by $E_w$ itself.

In the following theorem, we determine the set of bounded point evaluations for a commuting operator-valued multishift with invertible operator weights. 

\begin{proposition}\label{rep-ker}
Let $H$ be a complex separable Hilbert space and let $T = (T_1, \ldots, T_d)$ be a
commuting operator-valued multishift on $\ell^2_H(\mathbb N^d)$ with invertible operator
weights $\{A^{(j)}_{\alpha} : \alpha \in \mathbb N^d,\  j=1, \ldots, d\}$.
Let $\mathscr B = \{B_\alpha \in \mathcal G(H) : \alpha \in \mathbb N^d\}$,
where $B_\alpha$ is as defined in Proposition \ref{moments}. Then the set of all bounded point evaluations on $\mathcal H^2(\mathscr B)$ is contained in $\sigma_p(T^*)$ and is given by
\beqn
\Omega := \Big\{w \in \mathbb C^d : \sup_{x \in H, \|x\|=1} \sum_{\alpha \in \mathbb N^d} |w^\alpha|^2 \|B_\alpha^{*-1}x\|^2 < \infty\Big\}.
\eeqn
\end{proposition}

\begin{proof}
Let $w \in \mathbb C^d$ be a bounded point evaluation on $\mathcal H^2(\mathscr B)$.
 Then the evaluation map $E_w : \mathcal H^2(\mathscr B) \rar H$ is continuous. For $x \in H$, consider the vector $E_w^*x \in \mathcal H^2(\mathscr B)$ and note that for any polynomial $p \in \mathcal H^2(\mathscr B)$ and $j \in \{1, \ldots, d\}$,
\beq\label{eigenvector}
\inp{\mathscr M_{z_j}^* E_w^*x}{p}_{\!_{\mathcal H^2(\mathscr B)}} &=& \inp{E_w^*x}{z_j p}_{\!_{\mathcal H^2(\mathscr B)}} = \inp{x}{E_w(z_j p)}_{\!_H} = \inp{x}{w_j p(w)}_{\!_H}\notag\\
&=& \inp{\overline{w}_jx}{p(w)}_{\!_H} = \inp{\overline{w}_j E_w^* x}{p}_{\!_{\mathcal H^2(\mathscr B)}}.
\eeq
Since the polynomials are dense in $\mathcal H^2(\mathscr B)$, it follows that $\mathscr M_{z_j}^* E_w^*x = \overline{w}_j E_w^* x$ for all $j \in \{1, \ldots, d\}$. This shows that the set of all bounded point evaluations is contained in
 $\overline{\sigma_p(\mathscr M_z^*)}$ (the complex conjugate of $\sigma_p(\mathscr M_z^*)$). Since $\sigma_p(\mathscr M_z^*)$ has
 polycircular symmetry (see Corollary \ref{same-spectrum}), it follows that the set of all bounded point evaluations is contained in $\sigma_p(\mathscr M_z^*)$. For the remaining part, we first claim that if $w \in \mathbb C^d$ is a bounded point evaluation, then 
 \beqn
 f(z) := \sum_{\alpha \in \mathbb N^d} \overline w^\alpha (B_\alpha^{*} B_\alpha)^{-1} x z^\alpha \in \mathcal H^2(\mathscr B). 
 \eeqn
To see this, let $g(z) = \sum_{\alpha \in \mathbb N^d} x_\alpha z^\alpha \in \mathcal H^2(\mathscr B)$. For $n \in \mathbb N$, set 
\beqn
f_n(z) := \sum_{\underset{|\alpha| \leqslant n}{\alpha \in \mathbb N^d}} \overline w^\alpha (B_\alpha^{*} B_\alpha)^{-1} x z^\alpha \quad \mbox{and} \quad g_n(z) := \sum_{\underset{|\alpha| \leqslant n}{\alpha \in \mathbb N^d}} x_\alpha z^\alpha.
\eeqn
Then for $n \in \mathbb N$, we get
\beq\label{fn}
\inp{f_n(z)}{g(z)}_{\!_{\mathcal H^2(\mathscr B)}} \!\! &=&\!\! \Big\langle \sum_{\underset{|\alpha| \leqslant n}{\alpha \in \mathbb N^d}} \overline w^\alpha (B_\alpha^{*} B_\alpha)^{-1} x z^\alpha, \sum_{\alpha \in \mathbb N^d} x_\alpha z^\alpha\Big\rangle_{\!_{\mathcal H^2(\mathscr B)}}\notag\\
&=& \sum_{\underset{|\alpha| \leqslant n}{\alpha \in \mathbb N^d}} \inp{\overline w^\alpha B_\alpha^{*-1} x}{B_\alpha x_\alpha}_{\!_H}
= \Big\langle x,\ \sum_{\underset{|\alpha| \leqslant n}{\alpha \in \mathbb N^d}} w^\alpha x_\alpha \Big\rangle_{\!_H} = \inp{x}{E_w g_n}_{\!_H}\notag\\
&=&\inp{E_w^* x}{g_n}_{\!_{\mathcal H^2(\mathscr B)}}.
\eeq
This, in turn, implies that 
\beqn
\|f_n\|_{\!_{\mathcal H^2(\mathscr B)}} = \sup_{\underset{\|g\|=1}{g \in \mathcal H^2(\mathscr B)}}|\inp{f_n(z)}{g(z)}_{\!_{\mathcal H^2(\mathscr B)}}| \leqslant \|E_w^* x\|_{\!_{\mathcal H^2(\mathscr B)}} \quad \mbox{for all } n \in \mathbb N.
\eeqn
Thus $f \in \mathcal H^2(\mathscr B)$ and hence, the claim stands verified. Further, taking $n \rar \infty$ on both sides of \eqref{fn}, we get
\beq\label{Ew*}
E_w^* x = \sum_{\alpha \in \mathbb N^d} \overline w^\alpha (B_\alpha^{*} B_\alpha)^{-1} x z^\alpha \mbox{ for all } x \in H.
\eeq
Now the continuity of $E_w^*$ gives that
\beqn
 \sup_{x \in H, \|x\|=1} \|E_w^* x\|^2 \overset{\eqref{Ew*}}= \sup_{x \in H, \|x\|=1} \sum_{\alpha \in \mathbb N^d} |w^\alpha|^2 \|B_\alpha^{*-1}x\|^2 < \infty.
\eeqn
Thus the set of all bounded point evaluations is contained in $\Omega$.

Conversely, suppose that $w \in \Omega$. Then for each $x \in H$,
\beqn
g_{w,x} (z) := \sum_{\alpha \in \mathbb N^d} \overline w^\alpha B_\alpha^{-1} B_\alpha^{*-1} x z^\alpha \in \mathcal H^2(\mathscr B).
\eeqn
Now define $F_w : H \rar \mathcal H^2(\mathscr B)$ by $F_w x = g_{w,x} (z)$ for all $x \in H$.
It follows from the definition of $\Omega$ that $F_w$ is a bounded linear map. Further, for any polynomial $p \in \mathcal H^2(\mathscr B)$ and $x \in H$,
\beqn
\inp{F_w^* p}{x}_H = \inp{p}{F_w x}_{\mathcal H^2(\mathscr B)} = \inp{p}{g_{w,x} (z)}_{\mathcal H^2(\mathscr B)} = \inp{p(w)}{x}_H.
\eeqn
Thus $F_w^* = E_w$ and hence $w$ is a bounded point evaluation. This completes the proof of the proposition.
\end{proof}

\begin{remark}\label{proper-inclusion}
\hspace{-1cm}\begin{enumerate}
\item[(i)] An alternate verification of \eqref{Ew*} can be seen as follows: Since $0$ is always a bounded point evaluation, it is easy to see that $E_0^* x = x$ for all $x \in H$. Suppose that $w \in \mathbb C^d$ is a bounded point evaluation. Then for all $x, y \in H$, we get
\beqn
\inp{E_0 E_w^* x}{y}_{\!_H} = \inp{E_w^* x}{E_0^* y}_{\!_{\mathcal H^2(\mathscr B)}} = \inp{x}{y}_{\!_H}.
\eeqn
This shows that $E_0 E_w^* x = x$ for all $x \in H$. Since $E_w^*x$ is an eigenvector of $\mathscr M_z^*$ corresponding to eigenvalue $\overline{w}$, it follows from the proof of Theorem \ref{H2-beta-model}(ii) that there exists a non-zero vector $y \in H$ such that
 \beqn
 E_w^* x = \sum_{\alpha \in \mathbb N^d} \overline w^\alpha (B_\alpha^{*} B_\alpha)^{-1} y z^\alpha.
 \eeqn
Using $E_0 E_w^* x = x$, we get $y = x$. This verifies \eqref{Ew*}.
\item[(ii)]Unlike the classical case (\cite[Proposition 19]{JL}), the inclusion in the preceding proposition may be strict in general. For example, consider
  the commuting operator-valued multishift $T = (T_1, \ldots, T_d)$ on
  $\ell^2_{\mathbb C^2}(\mathbb N^d)$ with operator weights given by
  $A^{(j)}_{\alpha} = \begin{pmatrix} 2 & 0 \\ 0 & \frac{1}{2} \end{pmatrix}$ for
  all $\alpha \in \mathbb N^d$ and $j = 1, \ldots, d$. Then it is easy to see that
  $B_\alpha = \begin{pmatrix} 2^{|\alpha|} & 0 \\ 0 & \frac{1}{2^{|\alpha|}}
   \end{pmatrix}$ for all $\alpha \in \mathbb N^d$. Let $e_1 = (1,0)$ and $e_2 = (0,1)$
   be two orthonormal vectors in $\mathbb C^2$. Then observe that for $w = (1,\ldots, 1) \in \mathbb C^d$,
\beqn
\sum_{\alpha \in \mathbb N^d} |w^\alpha|^2 \|{B_\alpha^*}^{-1}e_1\|^2 <\infty
\mbox{ but } \sum_{\alpha \in \mathbb N^d} |w^\alpha|^2 \|{B_\alpha^*}^{-1}e_2\|^2 = \infty.
\eeqn
Thus by Theorem \ref{H2-beta-model}(ii), $w \in \sigma_p(T^*)$ but $w$ is not a bounded point evaluation.
\end{enumerate}
\end{remark}

The observation in the above remark motivates us for the following proposition. 

\begin{proposition}
Let $n$ be a positive integer and $T = (T_1, \ldots, T_d)$ be a commuting operator-valued multishift on $\ell^2_{\mathbb C^n}(\mathbb N^d)$ with invertble operator weights given by 
\beqn
A^{(j)}_{\alpha} := \begin{pmatrix} w_{1,\alpha}^{(j)} & 0 & \ldots & 0 \\ 0 & w_{2,\alpha}^{(j)} & \ldots & 0\\
\vdots & \vdots & \ddots & \vdots\\
0 & 0& \ldots & w_{n,\alpha}^{(j)} \end{pmatrix},
\eeqn
where $w_{k,\alpha}^{(j)} \in \mathbb C$ for all $\alpha \in \mathbb N^d$, $j = 1, \ldots, d$ and $k = 1, \ldots, n$.
Let $\mathscr B = \{B_\alpha \in \mathcal G(H) : \alpha \in \mathbb N^d\}$,
where $B_\alpha$ is as defined in Proposition \ref{moments}. Let $\Omega$ be the set of all bounded point evaluations on $\mathcal H^2(\mathscr B)$.
Then the following statements hold:
\begin{enumerate}
\item[(i)] $T$ is unitarily equivalent to $W_1 \oplus \cdots \oplus W_n$, where for each $k \in \{1, \ldots, n\}$, $W_k$ is the classical multishift on $\ell^2(\mathbb N^d)$ with weights $\{w_{k,\alpha}^{(j)} : \alpha \in \mathbb N^d,\ j=1,\ldots, d\}$.  
\item[(ii)] $\mathcal H^2(\mathscr B) = \mathcal H^2(\mathscr B_1) \oplus \cdots \oplus\mathcal H^2(\mathscr B_n)$, where for $k \in \{1, \ldots, n\}$, $\mathscr B_k = \{B_{k,\alpha} : \alpha \in \mathbb N^d\}$ and $B_{k,\alpha}$ is given by the expression of $B_\alpha$ in Proposition \ref{moments} by replacing $A^{(j)}_{\alpha}$ with $w_{k,\alpha}^{(j)}$.
\item[(iii)] $\Omega = \cap_{k=1}^n \sigma_p(W_k^*)$. 
\end{enumerate} 
\end{proposition}

\begin{proof}
Part (i) has already been established in \cite[Proposition 3.4]{GKT}. The conclusion in (ii) follows from the observations that $\ell^2_{\mathbb C^n}(\mathbb N^d) = \oplus_{k=1}^n  \ell^2(\mathbb N^d)$ and $B_\alpha = \oplus_{k=1}^n B_{k,\alpha}$. To see (iii), suppose that $w \in \Omega$. Then by Theorem \ref{rep-ker}, we have 
\beqn
\sup_{x \in \mathbb C^n, \|x\|=1} \sum_{\alpha \in \mathbb N^d} |w^\alpha|^2 \|B_\alpha^{*-1}x\|^2 < \infty.
\eeqn
Hence it follows from (ii) that 
\beqn
\sum_{\alpha \in \mathbb N^d} |w^\alpha|^2 |B_{k,\alpha}|^{-2} < \infty \mbox{ for all } k=1, \ldots, n.
\eeqn
Therefore, by Theorem \ref{H2-beta-model}(ii) (see also \cite[pp. 220]{JL}), we get that $w \in \sigma_p(W_k^*)$ for all $k = 1, \ldots, n$. This proves that $\Omega \subseteq \cap_{k=1}^n \sigma_p(W_k^*)$. The proof for the reverse inclusion is similar.
\end{proof}

Let $f = \sum_{\alpha \in \mathbb N^d} x_\alpha z^\alpha \in \mathcal H^2(\mathscr B)$ and $\{p_n\}_{n \in \mathbb N}$ be a sequence of polynomials converging to $f$ in $\mathcal H^2(\mathscr B)$. Let $\Omega$ be the set of all bounded point evaluations on $\mathcal H^2(\mathscr B)$. If $w \in \Omega$, then we define $f(w)$ by
\beqn
f(w) := E_w f = \lim_{n \rar \infty} E_w p_n = \lim_{n \rar \infty} p_n(w).
\eeqn 
Since $f(w)$ is well-defined for all $w \in \Omega$, $f$ defines an $H$-valued function on $\Omega$. In other words, if $f \in \mathcal H^2(\mathscr B)$, then $f|_\Omega$ is an $H$-valued function on $\Omega$, where the restriction of $f$ on $\Omega$ should be interpreted in above sense. Consider the inner product space $\mathcal H(\kappa) := \{f|_\Omega : f \in \mathcal H^2(\mathscr B)\}$, with the inner product inherited from $\mathcal H^2(\mathscr B)$. Then it is easy to see that $\mathcal H(\kappa)$ is a Hilbert space of $H$-valued functions on $\Omega$, and $H$-valued polynomials are dense in $\mathcal H(\kappa)$. Further,  it follows from \cite[Chapter 6]{PR} that $\mathcal H(\kappa)$ is a reproducing kernel Hilbert space with the reproducing kernel $\kappa : \Omega \times \Omega \rar \mathcal B(H)$ given by
\beq\label{rk}
\kappa(z,w) = E_z E_w^* \overset{\eqref{Ew*}}= \sum_{\alpha \in \mathbb N^d} \big(B_\alpha^* B_\alpha\big)^{-1} z^\alpha \overline{w}^\alpha, \quad z, w \in \Omega.\eeq

It turns out that if $\Omega$ has non-empty interior, then the elements of $\mathcal H(\kappa)$ define $H$-valued holomorphic functions on the interior of $\Omega$. Indeed, we have the following proposition.

\begin{proposition}\label{holo}
Let $\mathcal H(\kappa)$ be defined as above.
If $\Omega$ has a non-empty interior, then $\mathcal H(\kappa)$ is a Hilbert space of $H$-valued holomorphic functions on the interior of $\Omega$.
\end{proposition}

\begin{proof}
Let $w=(w_1,\ldots,w_d)$ be any point in the interior of $\Omega$. Then there exists $\tilde{w}=(\tilde{w_1},\ldots,\tilde{w_d})$ in the interior of $\Omega$ such that $|w_j|<|\tilde{w_j}|$ for all $j=1,\ldots,d$. Let $\phi:\mathbb N\rar\mathbb N^d$ be a bijective map and $f=\sum_{\alpha \in \mathbb N^d} x_\alpha z^\alpha\in \mathcal H(\kappa)$. Note that the sequence of polynomials $\{p_n\}_{n\in\mathbb N}$ converges to $f$ in $\mathcal H(\kappa)$, where for each $n\in\mathbb N$,
\beqn
 p_n(z) =\sum_{j=1}^{n} x_{\phi(j)} z^{\phi(j)}, \quad z \in \Omega.
\eeqn
Let $\epsilon>0$. Since $\{E_{\tilde{w}} p_n\}_{n\in\mathbb N^d}$ is a Cauchy sequence in $H$, there exists $n_0\in\mathbb N$ such that $\|E_{\tilde{w}}p_n-E_{\tilde{w}}p_{n+1}\|=\|x_{\phi(n+1)}\| |\tilde{w}^{\phi(n+1)}| < \epsilon$ for all $n\geqslant n_0.$
This proves that $\{\tilde{w}^\alpha x_\alpha\}_{\alpha\in\mathbb N^d}$ is a bounded sequence in $H$. Therefore, we have
\beqn
\sum_{\alpha\in\mathbb N^d}\|w^\alpha x_\alpha\|  = \sum_{\alpha\in\mathbb N^d}\|\tilde{w}^\alpha x_\alpha\|  \frac{|w^\alpha|}{|\tilde{w}^\alpha|}\leqslant \sup_{\alpha\in\mathbb N^d} \|\tilde{w}^\alpha x_\alpha\|  \sum_{\alpha\in\mathbb N^d} \frac{|w^\alpha|}{|\tilde{w}^\alpha|}<\infty.
\eeqn
Thus $f$ converges absolutely at $w$ and hence is holomorphic on the interior of $\Omega.$ This completes the proof of proposition.
\end{proof}

The following theorem shows that the operator-valued multishifts can be realized as the tuple of operators of multiplication by the coordinate functions on a reproducing kernel Hilbert space of vector valued holomorphic functions, provided the set of bounded point evaluations has non-empty interior.

\begin{theorem}\label{T-H(k)}
Let $H$ be a complex separable Hilbert space and $T = (T_1, \ldots, T_d)$ be a
commuting operator-valued multishift on $\ell^2_H(\mathbb N^d)$ with invertible operator
weights $\{A^{(j)}_{\alpha} : \alpha \in \mathbb N^d,\  j=1, \ldots, d\}$.
Let $\mathscr B = \{B_\alpha \in \mathcal G(H) : \alpha \in \mathbb N^d\}$,
where $B_\alpha$ is as defined in Proposition \ref{moments}. Let $\Omega$ be the set of all bounded point evaluations on $\mathcal H^2(\mathscr B)$. Then for every $w \in \Omega,\; \ker{(\mathscr M_z^* - \overline{w})}=
\{\kappa(\cdot, w) x : x \in H \}$ and \beq\label{span-closure} \bigvee_{w \in
\Omega}\ker{(\mathscr M_z^* - \overline{w})}= \mathcal H(\kappa).\eeq
Moreover, if $\Omega$ has non-empty interior, then $T= (T_1, \ldots, T_d)$ is unitarily equivalent to the $d$-tuple $\mathscr M_z = (\mathscr M_{z_1}, \ldots, \mathscr M_{z_d})$ of multiplication operators by the coordinate functions on $\mathcal H(\kappa)$.
\end{theorem}

\begin{proof}
Fix $w \in \Omega$. It follows from the general theory of reproducing kernel Hilbert space  that for each non-zero $x \in H$,  $\kappa(\cdot, w)x$ is an eigenvector for $\mathscr M_z^*$ corresponding to the eigenvalue $\overline{w}$. Hence, $\{\kappa(\cdot, w)x : x \in H \}
\subseteq \ker{(\mathscr M_z^* - \overline{w})}.$ Further, suppose that $f \in
\ker{(\mathscr M_z^* - \overline{w})} \subseteq \mathcal H(\kappa)$. Then $f = g|_\Omega$ for some $g \in \mathcal H^2(\mathscr B)$. For an $H$-valued polynomial $p$ and $j \in \{1, \ldots, d\}$, we get
\beqn
\inp{\overline{w}_j g}{p}_{\!_{\mathcal H^2(\mathscr B)}} &=& \inp{\overline{w}_j f}{p}_{\!_{\mathcal H(\kappa)}} = \inp{\mathscr M^*_{z_j} f}{p}_{\!_{\mathcal H(\kappa)}} = \inp{f}{z_j p}_{\!_{\mathcal H(\kappa)}} = \inp{g}{z_j p}_{\!_{\mathcal H^2(\mathscr B)}}\\ 
&=& \inp{\mathscr M^*_{z_j} g}{p}_{\!_{\mathcal H^2(\mathscr B)}}.
\eeqn
Thus $g \in \ker{(\mathscr M_z^* - \overline{w})} \subseteq \mathcal H^2(\mathscr B)$. It follows from
the proof of Theorem \ref{H2-beta-model}(ii) that there exists a non-zero
vector $x$ in $H$ such that 
\beqn g(z) = \displaystyle \sum_{\alpha
\in \mathbb N^d} \overline{w}^\alpha (B_\alpha^* B_\alpha)^{-1} x z^\alpha. 
\eeqn 
Since $g|_\Omega  = f$, $f = \kappa(\cdot, w)x$. This shows that $\{\kappa(\cdot, w)x : x \in H \} = \ker{(\mathscr M_z^* - \overline{w})}.$ The conclusion in \eqref{span-closure} follows from
the reproducing property of $\mathcal H(\kappa).$

For the moreover part, in view of Theorem \ref{H2-beta-model}, it is enough to show that $\mathscr M_z = (\mathscr M_{z_1}, \ldots, \mathscr M_{z_d})$ on $\mathcal H^2(\mathscr B)$ is unitarily equivalent to  $\mathscr M_z = (\mathscr M_{z_1}, \ldots, \mathscr M_{z_d})$ on $\mathcal H(\kappa).$ To this end, suppose that $\Omega$ has non-empty interior. Define $U : \mathcal H^2(\mathscr B) \rar \mathcal H(\kappa)$ by
\beqn
U f = f|_{\Omega}, \quad f \in \mathcal H^2(\mathscr B).
\eeqn
Since $\Omega$ has non-empty interior, in the view of Proposition \ref{holo}, $U$ is injective. As the inner product in $\mathcal H(\kappa)$ is inherited from $\mathcal H^2(\mathscr B)$, it follows that $U$ is a surjective isometry. Now for $f \in \mathcal H^2(\mathscr B)$ and $j \in \{1, \ldots, d\}$, we have
\beqn
U \mathscr M_{z_j} f = U(z_j f) = (z_j f)|_\Omega = \mathscr M_{z_j} f|_\Omega = \mathscr M_{z_j} U f.
\eeqn
This completes the proof of the theorem.
\end{proof}

Let $\Omega$ be a subset of $\mathbb C^d$ with non-empty interior. Let $\mathcal H(\kappa)$ be a reproducing kernel Hilbert space of vector-valued functions on $\Omega$ which are holomorphic on the interior of $\Omega$ with reproducing kernel $\kappa$. Suppose that $\mathscr M_{z_j}$, the multiplication operator by the coordinate function $z_j$, on $\mathcal H(\kappa)$ is bounded for all $j=1,\ldots,d$. Then it is easy to see that  the $d$-tuple $\mathscr M_z=(\mathscr M_{z_1},\ldots,\mathscr M_{z_d})$ on $\mathcal H(\kappa)$ is unitarily equivalent to the $d$-tuple $\mathscr M_z=(\mathscr M_{z_1},\ldots,\mathscr M_{z_d})$ on $\mathcal H(\kappa|_{\Omega_0})$ for all non-empty open subset 
$\Omega_0$ of $\Omega.$ 
Let $\Omega_1,\Omega_2\subseteq \mathbb C^d$ be such that the interior of $\Omega_1\cap \Omega_2$ is non-empty.
Let  $\mathcal H(\kappa_1)$ and $\mathcal H(\kappa_2)$ be two reproducing kernel Hilbert spaces of vector-valued functions on $\Omega_1$ and $\Omega_2$ respectively, which are holomorphic on the respective interiors. Thus, in order to characterize the unitary equivalence of $\mathscr M_z$ on $\mathcal H(\kappa_1)$ and $\mathcal H(\kappa_2)$, there is no loss of generality if we assume $\Omega_1 = \Omega_2$. In view of this, we have the following theorem.
 
 \begin{theorem}\label{unitary-criteria}
Let $H$ be a complex separable Hilbert space and $T = (T_1, \ldots, T_d)$, $\tilde{T} = (\tilde{T_1}, \ldots, \tilde{T_d})$ be two commuting operator-valued multishifts on $\ell^2_{H}(\mathbb N^d)$  with respective invertible operator weights $\{A^{(j)}_{\alpha}: \alpha \in \mathbb N^d, \  j=1, \ldots, d\}$ and $\{\tilde{A}^{(j)}_{\alpha}: \alpha \in \mathbb N^d, \  j=1, \ldots, d\}$. Let $\mathscr B = \{B_\alpha \in \mathcal G(H) : \alpha \in \mathbb N^d\}$ and $\tilde {\mathscr B} = \{\tilde B_\alpha \in \mathcal G(H) : \alpha \in \mathbb N^d\}$, where $B_\alpha$ and $\tilde B_\alpha$ are as defined in Proposition \ref{moments} corresponding to $A^{(j)}_{\alpha}$ and $\tilde{A}^{(j)}_{\alpha}$ respectively. Let $\Omega$ be the set of bounded point evaluations on $\mathcal H^2(\mathscr B)$ and $\mathcal H^2(\tilde{\mathscr B})$ with non-empty interior. Then $T$ and $\tilde{T}$ are unitarily equivalent if and only if there exists a unitary operator $U$ on $H$ such that 
\beq \label{unitary-eq}
U B_\alpha^* B_\alpha = \tilde B_\alpha^* \tilde B_\alpha U \mbox{ for all } \alpha \in \mathbb N^d.
\eeq
\end{theorem}

\begin{proof}
It follows from Theorem \ref{T-H(k)} that $T$ is unitarily equivalent to $\mathscr M_z$ on $\mathcal H(\kappa)$ and $\tilde T$ is unitarily equivalent to $\mathscr M_z$ on $\mathcal H(\tilde\kappa)$. Hence, we show that $\mathscr M_z$ on $\mathcal H(\kappa)$  is unitarily equivalent to $\mathscr M_z$ on $\mathcal H(\tilde\kappa)$ if and only if \eqref{unitary-eq} holds. To this end, suppose that $\mathscr M_z$ on $\mathcal H(\kappa)$ is unitarily equivalent to $\mathscr M_z$ on $\mathcal H(\tilde \kappa)$. Then there exists a unitary operator ${\mathscr U} : \mathcal H(\kappa) \rar \mathcal H(\tilde \kappa)$ such that ${\mathscr U} \mathscr M_{z_j} = \mathscr M_{z_j} {\mathscr U}$ for all $j = 1, \ldots, d.$ Therefore by \cite[Theorem 3.7]{C-S}, ${\mathscr U} = \mathscr M_\Phi$ for some $\Phi : \Omega \rar \mathcal B(H)$. Since both ${\mathscr U}$ and ${\mathscr U}^*$ intertwine $\mathscr M_z$, following the arguments of the proof of \cite[Theorem 8]{Chen}, we see that $\Phi$ is constant and satisfies $\Phi(z) \kappa(z,w) =\tilde \kappa(z,w) \Phi(z)$, $z,w\in \Omega$. Let $\Phi(z) = U$ for all $z \in \Omega$. Then it is easy to see that $U$ is a unitary operator on $H$ and
\beqn
\tilde \kappa(z,w) = U \kappa(z,w) U^*, \quad z,w \in \Omega.
\eeqn
Now for $x, y \in H$ and $z,w \in \Omega$, we get
\beqn
\Big\langle \sum_{\alpha \in \mathbb N^d} \big(\tilde B_\alpha^* \tilde B_\alpha\big)^{-1} x z^\alpha \overline{w}^\alpha, y \Big\rangle_{\!_{H}} &\overset{\eqref{rk}} =& \inp{\tilde \kappa(z,w)x}{y}_{\!_{H}} = \inp{U \kappa(z,w) U^* x}{y}_{\!_{H}}\\
&=&\inp{\kappa(\cdot, w)U^*x}{\kappa(\cdot, z)U^*y}_{\!_{\mathcal H(\kappa)}}\\
&\overset{\eqref{rk}}=& \Big\langle \sum_{\alpha \in \mathbb N^d} U \big(B_\alpha^*  B_\alpha\big)^{-1}U^* x\, z^\alpha \overline{w}^\alpha, y \Big\rangle_{\!_{H}}.
\eeqn
Thus \eqref{unitary-eq} holds.

Conversely, suppose that there exists a unitary operator $U$ on $H$ such that \eqref{unitary-eq} holds. Equivalently, as noted above, we have
\beqn
\tilde \kappa(z,w) = U \kappa(z,w) U^*, \quad z,w \in \Omega.
\eeqn
Define $\Phi : \Omega \rar \mathcal B(H)$ by
\beqn
\Phi(z) = U, \quad z \in \Omega.
\eeqn
Then from \cite[Theorem 6.28]{PR}, the map $\mathscr M_\Phi : \mathcal H(\kappa) \rar \mathcal H(\tilde \kappa)$ defined by 
\beqn
(\mathscr M_\Phi f)(z) = \Phi(z) f(z) = U f(z), \quad f \in \mathcal H(\kappa), \ z \in \Omega. 
\eeqn
is bounded. It is easy to see that $\mathscr M_{\Phi} \mathscr M_{z_j} = \mathscr M_{z_j} \mathscr M_{\Phi}$ for all $j = 1, \ldots, d.$ Using \cite[Theorem 3.7]{C-S}, for $z, w \in \Omega$ and $x \in H$, we get  
\beqn 
(\mathscr M_{\Phi} \mathscr M^*_{\Phi} \tilde \kappa(\cdot,w) x) (z)
= \Phi(z) \kappa(z,w) \Phi(w)^* x 
=U \kappa(z,w) U^* x
= \tilde \kappa(z,w) x .
\eeqn
This shows that $\mathscr M^*_{\Phi}$ can be extended isometrically on $\mathcal H(\tilde \kappa).$ Since $\Phi(w)= U$ is unitary for all $w \in \Omega$, range of $\mathscr M^*_{\Phi}$ is
\beqn 
\bigvee_{w \in
\Omega} \{\kappa(\cdot, w)\Phi(w)^* x : x \in H \}= \mathcal H(\kappa). 
\eeqn
Therefore $\mathscr M_{\Phi}$ is unitary.
This completes the proof.
\end{proof}

\medskip \textit{Acknowledgment}.
The authors are grateful to Gadadhar Misra, Sameer Chavan and Soumitra Ghara for several helpful suggestions and constant support. We express our gratitude to the faculty and the administration of Department of Mathematics and Statistics, IIT Kanpur and Department of Mathematics, IISc Bangalore for their warm hospitality during the preparation of this paper. We convey our sincere thanks to Jan Stochel for several useful comments.

\end{document}